\documentclass[11pt,reqno,a4paper]{article}
\usepackage[top=0.8in, bottom=0.8in, left=0.8in, right=0.8in]{geometry}
\usepackage{comment}
\usepackage{pgfplots}
\pgfplotsset{compat=1.18}

\usepackage{indentfirst}
\usepackage[utf8]{inputenc}
\usepackage{amsmath,amsfonts,amssymb,graphicx,csquotes,mathtools}
\usepackage{stackengine,bbm,mathrsfs}
\usepackage{parskip}
\usepackage{authblk}

\usepackage{hyperref}
\usepackage{color}
\usepackage[dvipsnames]{xcolor}

\usepackage{graphicx}
\graphicspath{./images}
\usepackage{subcaption}

\usepackage[colorinlistoftodos]{todonotes}
\usepackage{tikz-cd}
\usepackage{amsthm}
\theoremstyle{definition}
\newtheorem{defn}{\newline\protect\definitionname}[section]
\theoremstyle{plain}
\newtheorem{lem}[defn]{\newline\protect\lemmaname}

\theoremstyle{plain}
\newtheorem{thm}[defn]{\newline\protect\theoremname}
\theoremstyle{plain}
\newtheorem*{thm*}{Theorem}
\newtheorem{prop}[defn]{\newline\protect\propositionname}
\theoremstyle{plain}
\newtheorem{cor}[defn]{\newline\protect\corollaryname}
\theoremstyle{remark}
\newtheorem{rem}[defn]{\newline\protect\remarkname}
\theoremstyle{definition}

\makeatother
\providecommand{\definitionname}{Definition}
\providecommand{\examplename}{Example}
\providecommand{\lemmaname}{Lemma}
\providecommand{\propositionname}{Proposition}
\providecommand{\remarkname}{Remark}
\providecommand{\corollaryname}{Corollary}
\providecommand{\theoremname}{Theorem}
\numberwithin{defn}{section}

\newcommand{\quand}{\quad\text{ and } \quad}
\usepackage{xcolor}

\makeatletter
\def\namedlabel#1#2{\begingroup
    #2%
    \def\@currentlabel{#2}%
    \phantomsection\label{#1}\endgroup
}

\title{Quenched Mixing Rates for Doubly Intermittent Maps }
\author[1]{Mubarak Muhammad \footnote{\href{mailto:mubarak@aims.edu.gh}{\texttt{mubarak@aims.edu.gh}}}}
\author[2]{Marks Ruziboev \footnote{\href{mailto:    marx.ruziboev@gmail.com }{\texttt{    marx.ruziboev@gmail.com}}\\{\texttt{ Mubarak Muhammad is a Faculty Postdoctoral Fellow, School of Mathematical Sciences, Hebei Normal University, Shijiazhuang, Hebei Province, 050024 China.}}}}

\affil[1]{School of Mathematical Sciences, Hebei Normal University, Shijiazhuang, Hebei Province, 050024 China}
\affil[2]{V.I.Romanovskiy Institute of Mathematics, Academy of Sciences of the Republic of Uzbekistan, Tashkent, Uzbekistan}

\date{}
\begin{document}

\maketitle
\begin{abstract}

We study quenched mixing rates for random compositions of two classes of interval
maps with two indifferent fixed points and a singularity at the origin: Pikovsky maps
and Grossmann--Horner maps. For the Pikovsky family, each fibre map preserves
Lebesgue measure, so the equivariant sample measures are given by
\(\mu_\omega=m\). For the Grossmann--Horner family, we construct an equivariant
family \((\mu_\omega)_{\omega\in\Omega}\) of absolutely continuous probability
measures. Using random Young towers, we prove quenched polynomial decay of both
future and past fibre correlations for bounded observables against H\"older
observables. The rates are determined by quenched return time tail estimates obtained
from endpoint drift bounds for the random cocycle.
\end{abstract}
\begin{figure}[ht]
\centering
\makebox[\textwidth][c]{%
\begin{subfigure}{.48\textwidth}
\centering
\begin{tikzpicture}
\begin{axis}[
    width=7cm,
    height=7cm,
    xmin=0, xmax=1,
    ymin=0, ymax=1,
    axis lines=box,
    ticks=none,
    clip=false,
]

\addplot[
    gray,
    dashed,
    domain=0:1,
    samples=2
] {x};

\addplot[
    red,
    thick,
    domain=0:0.499,
    samples=200
] {x*(1+(2*x)^0.5)};

\addplot[
    red,
    thick,
    domain=0.501:1,
    samples=200
] {-(1-x)*(1+(2*(1-x))^0.5)+1};

\end{axis}
\end{tikzpicture}
\caption{Pikovsky Map}
\label{graph_g}
\end{subfigure}%
\hspace{-1.15cm}
\begin{subfigure}{.48\textwidth}
\centering

\begin{tikzpicture}
\begin{axis}[
    width=7cm,
    height=7cm,
    xmin=0, xmax=1,
    ymin=0, ymax=1,
    axis lines=box,
    ticks=none,
    clip=false,
]

\addplot[
    gray,
    dashed,
    domain=0:1,
    samples=2
] {x};

\addplot[
    Red,
    thick,
    domain=0:0.499,
    samples=200
] {x*(1+(2*x)^0.5)};

\addplot[
    red,
    thick,
    domain=0.501:1,
    samples=200
] {(1-x)*(1+(2*(1-x))^0.5)};

\end{axis}
\end{tikzpicture}
\caption{Grossmann--Horner Map}
\label{graph_h}
\end{subfigure}%
}
\end{figure}
\section{Setup and Notation}\label{sec:setup}

We introduce the class of random dynamical systems considered in the paper and fix the basic notation.

\medskip

Let \(X\subset\mathbb{R}\) be a compact interval and let
\(
\mathcal{F}=\{f_t:X\to X \mid t\in\mathcal{A}\}
\)
be a parametrized family of maps. The parameter space \(\mathcal{A}\) is equipped with a probability measure \(\nu\). We define
\[
\Omega=\mathcal{A}^{\mathbb{Z}},
\qquad
\mathbb{P}=\nu^{\mathbb{Z}},
\]
and let \(\sigma:\Omega\to\Omega\) be the left shift. Thus
\((\Omega,\mathcal B(\Omega),\mathbb P,\sigma)\)
is an invertible measure preserving system.

For \(\omega=(\omega_j)_{j\in\mathbb Z}\in\Omega\), set \(f_\omega=f_{\omega_0}\). The associated cocycle is
\[
f_\omega^n
=
f_{\sigma^{n-1}\omega}\circ\cdots\circ f_\omega,
\qquad n\ge1,
\]
with \(f_\omega^0=\mathrm{Id}\). The skew product
\(
S(\omega,x)=(\sigma\omega,f_\omega(x))
\)
satisfies \(S^n(\omega,x)=(\sigma^n\omega,f_\omega^n(x))\).

\medskip

\begin{defn}
A family of probability measures \(\{\mu_\omega\}_{\omega\in\Omega}\) on \(X\) is equivariant if
\(
(f_\omega)_*\mu_\omega=\mu_{\sigma\omega}
\quad \text{for }\mathbb P\text{-a.e. }\omega.
\)
\end{defn}

\medskip

We study statistical properties of the random fibre systems \((f_\omega,\mu_\omega)\), in particular decay of correlations.

\begin{defn}
For \(\varphi,\psi:\Omega\times X\to\mathbb R\), write
\[
\varphi_\omega(x)=\varphi(\omega,x),
\qquad
\psi_\omega(x)=\psi(\omega,x).
\]
For \(n\ge1\), define the future correlation
\[
\mathrm{Cor}^{(F)}_{n,\omega}(\varphi,\psi,X)
=
\int_X (\varphi_{\sigma^n\omega}\circ f_\omega^n)\psi_\omega\,d\mu_\omega
-
\int_X \varphi_{\sigma^n\omega}\,d\mu_{\sigma^n\omega}
\int_X \psi_\omega\,d\mu_\omega,
\]
and the past correlation
\[
\mathrm{Cor}^{(P)}_{n,\omega}(\varphi,\psi,X)
=
\int_X (\varphi_\omega\circ f_{\sigma^{-n}\omega}^n)\psi_{\sigma^{-n}\omega}\,d\mu_{\sigma^{-n}\omega}
-
\int_X \varphi_\omega\,d\mu_\omega
\int_X \psi_{\sigma^{-n}\omega}\,d\mu_{\sigma^{-n}\omega}.
\]
\end{defn}

Deterministic observables on \(X\) are identified with constant fibre families.

\medskip

\subsection*{Notation and conventions}

Throughout, \(\nu\) denotes the parameter law and \(m\) denotes Lebesgue measure.

The fibre space \(X\) is either
\[
X=\mathbb T
\quad \text{(Pikovsky case)},
\quad
or
\quad
X=I=[-1,1]
\quad \text{(Grossmann--Horner case)}.
\] where \(\mathbb T=I/\{-1\sim1\},  I=[-1,1].\)
We write \(I^-=[-1,0)\), \(I^+=(0,1]\).

For Pikovsky maps we write \(f_\omega=g_\omega\), and for Grossmann--Horner maps \(f_\omega=h_\omega\).

\medskip

\noindent
\textbf{Random towers.}
We use \(\Lambda_\omega\) for the base, \(\mathcal P_\omega\) for the base partition,
\(R_\omega\) for the return time, \(\Delta_\omega\) for the tower, and
\[
F_\omega:\Delta_\omega\to\Delta_{\sigma\omega}
\]
for the tower map. The projection \(\pi_\omega:\Delta_\omega\to X\) satisfies
\(
\pi_{\sigma\omega}\circ F_\omega=f_\omega\circ\pi_\omega.
\)
The separation time is denoted by \(s\), with \(s(z,z')=0\) when \(z,z'\) lie in different atoms.

We write \(\bar m_\omega\) for the reference measure on \(\Delta_\omega\),
\(\bar\mu_\omega\) for the equivariant tower measure, and
\[
\mu_\omega=(\pi_\omega)_*\bar\mu_\omega.
\]
In the Pikovsky case, \(\mu_\omega=m\).

Jacobians are taken with respect to \(\bar m_\omega\); in particular
\(JF_\omega^n\) denotes the Jacobian of \(F_\omega^n\), and on fibres
\(Jf_\omega^n=|Df_\omega^n|\).

\medskip

\noindent
Endpoint sequences \(x_n^\pm(\omega)\), \(y_n^\pm(\omega)\),
\(\Delta_n^\pm(\omega)\), \(\delta_n^\pm(\omega)\) are always fibre dependent.
When parameters are frozen, we write \(x_n^\pm(\alpha)\), etc.

We write \(A\lesssim B\) if \(A\le CB\) for a constant independent of \(n,\omega\)
(and parameters when stated). We write \(A\asymp B\) for two sided bounds.
Constants \(C\) may change from line to line; random constants are written \(C_\omega\).

We use the paired sign convention:
\[
1\mp x_n^\pm = 
\begin{cases}
1-x_n^+ & (+\text{ case})\\
1+x_n^- & (-\text{ case})
\end{cases}.
\]

\section{Introduction and Literature Review}\label{sec:intro}

Quenched statistical properties of random dynamical systems have been studied
by several methods. Early results on quenched decay of correlations were obtained
in \cite{Buzzi,BBMD02} using, respectively, functional analytic and inducing
techniques. Since then, quenched limit theorems and mixing estimates have been
proved in a number of random and non autonomous settings. Even for uniformly
hyperbolic systems, random compositions require non trivial modifications of the
deterministic arguments; see, for example, \cite{DFGV20}. Related results have
also been obtained for expanding on average systems
\cite{DH21,DHS21,DFGV20}.

For systems that are not uniformly hyperbolic, transfer operator methods are not
always directly applicable. Random Young towers provide a flexible framework for
studying quenched statistical properties in this setting. This approach has been
used in \cite{BBMD02,BBMD02C,ABR22,ABRV,BahBosRuz19} to obtain quenched decay
of correlations, and in \cite{Haf22,Su2022} to prove quenched limit theorems.
The construction of such towers in concrete examples, however, depends on the
geometry of the maps and on the way randomness enters the endpoint recursions.

We study random compositions of interval maps with two indifferent fixed points
and a singularity at the origin. The first family was introduced by
Pikovsky~\cite{pik91} and further studied in \cite{artcris04}; the second was
introduced by Grossmann and Horner~\cite{gh85}. Their deterministic ergodic and
statistical properties were studied in \cite{gnsp,CoaLuzMub22,MT22}.

These systems differ from the LSV maps in two important ways. First, the inducing
scheme must track excursions between two neutral endpoints through a singular
branch. Second, the singularity gives rise to unbounded derivatives, so the
bounded distortion estimates required for the induced branches do not follow
from the standard LSV construction. A central part of the paper is therefore the
verification of the distortion and return time tail estimates needed for the
random tower construction. For the distortion estimates, we treat the singular
and neutral parts of each excursion separately in distance coordinates. The
one step neutral contributions then telescope along the excursion, giving
uniform distortion bounds for the induced random branches without relying on a
Schwarzian or Koebe argument.

These estimates are not merely a formal verification of the abstract tower
hypotheses. Their relevance is also illustrated by subsequent work on memory
loss for nonstationary intermittent systems with unbounded derivatives; in
particular, Korepanov and Lepp\"anen~\cite{korepanov2024improved} use the
distortion estimates for the Pikovsky and Grossmann--Horner families developed
here.

The probabilistic estimates required for the drift conditions also depend on
the structure of the random family. In the random LSV setting of
\cite{BahBosRuz19}, the endpoint estimates can be reduced to concentration
bounds involving the independent parameter sequence. In the present setting,
the endpoint quantities are generated by the random cocycle and therefore
involve overlapping blocks of the environment. For both the Pikovsky and
Grossmann--Horner families, we bound the raw endpoint arrays from below by
one-coordinate comparison arrays. For each fixed row, these comparison
variables are independent and uniformly bounded, so Hoeffding's inequality
gives the required lower drift estimates with stretched exponential random
cutoffs.

Under the stated assumptions on the parameter law, we obtain equivariant
absolutely continuous sample measures and quenched polynomial decay of both
future and past fibre correlations. The resulting rates are determined by the
return time tail estimates obtained from the endpoint drift bounds.

\subsection{Family of Maps}\label{sec:main-results}

\subsubsection{Pikovsky maps}

Let \(I=[-1,1]\), \(I^-=[-1,0)\), \(I^+=(0,1]\), and
\(
\mathbb{T}=I/\{-1\sim1\}.
\)
We consider the family of maps
\[
g_\alpha:\mathbb{T}\to\mathbb{T},\qquad \alpha>1,
\]
introduced in \cite{pik91}. On the positive half of the interval, \(g_\alpha\)
is defined implicitly by
\begin{equation}\label{eqn_1}
x =
\begin{cases}
\dfrac{1}{2\alpha}(1+g_\alpha(x))^\alpha,
& 0\leq x\leq \dfrac{1}{2\alpha},\\[1.2ex]
g_\alpha(x)+\dfrac{1}{2\alpha}(1-g_\alpha(x))^\alpha,
& \dfrac{1}{2\alpha}\leq x\leq 1.
\end{cases}
\end{equation}
For negative \(x\), we set \(g_\alpha(-x)=-g_\alpha(x)\). Each branch defines a
unique monotone inverse, so \(g_\alpha\) is well defined and continuous on
\(\mathbb{T}\setminus\{0\}\), and \(C^2\) on \(I\setminus\{0\}\) (modulo
\(-1\sim1\)). For \(\alpha=1\), the map reduces to the doubling map.

For every \(\alpha>1\), \(g_\alpha\) preserves Lebesgue measure \(m\)
(see \cite{gnsp}), and in the deterministic setting it has polynomial decay of
correlations of order \(n^{-1/(\alpha-1)}\).

\medskip

Fix parameters
\(
1<\alpha_1<2,\quad \alpha_1<\alpha_2<\alpha_1+1,
\)
and set \(\mathcal A=[\alpha_1,\alpha_2]\).
We take \(X=\mathbb{T}\) and \(f_\alpha=g_\alpha\). Then
\[
\Omega=\mathcal A^{\mathbb Z},\qquad
\mathbb P=\nu^{\mathbb Z},\qquad
g_\omega=g_{\omega_0},
\]
and the cocycle is given by
\(
g_\omega^n
=
g_{\sigma^{n-1}\omega}\circ\cdots\circ g_\omega.
\)
The skew product
\(
S(\omega,x)=(\sigma\omega,g_\omega(x))
\)
preserves \(\mathbb P\times m\).

The parameter restrictions are used in two places: uniform distortion estimates
require bounded variation over \([\alpha_1,\alpha_2]\), while return time tails
require the summability condition
\(
\alpha_1/(\alpha_2-1)>1,
\)
equivalently \(\alpha_2<\alpha_1+1\).

\subsubsection{Grossmann--Horner maps}
\noindent
The second class of maps we consider was introduced in \cite{gh85}; see Figure~\ref{graph_h}. These maps are full branch maps
\(
h:I\to I,
\quad I=[-1,1],
\)
with two surjective branches defined on
\(
I^-=[-1,0),
\qquad
I^+=(0,1].
\)
The map \(h\) is \(C^1\) on \(I\setminus\{0\}\) and \(C^2\) on
\(
I\setminus(\{0\}\cup\{\pm1\}).
\)
It has a singularity at \(0\), indifferent fixed point at \(-1\) and a neutral point at \(1\). More precisely, for parameters
\(
\gamma>1,
\qquad
0<k<1,
\qquad
a,b>0,
\)
the local behaviour is
\begin{equation}\label{eqn_2}
\begin{cases}
h(x)=1-b|x|^k+o(|x|^k),
& \text{if } x\in U_0(\gamma,k), \\[0.5ex]
h(x)=-x+a|x-1|^\gamma+o(|x-1|^\gamma),
& \text{if } x\in U_{1_-}(\gamma,k), \\[0.5ex]
h(x)=x+a|x+1|^\gamma+o(|x+1|^\gamma),
& \text{if } x\in U_{-1_+}(\gamma,k).
\end{cases}
\end{equation}
Here \(U_0(\gamma,k)\) is a neighbourhood of \(0\),
\[
U_{1_-}(\gamma,k)=h(U_0(\gamma,k)\cap(-1,0))
\qquad
U_{-1_+}(\gamma,k)=h(U_0(\gamma,k)\cap(0,1))
\]
are a left neighbourhood of \(1\), and a right neighbourhood of \(-1\) respectively.

For \(i=1,2\), the corresponding derivative estimates are
\begin{equation}\label{eqn_3}
\begin{cases}
|D^i h(x)|
=
C|x|^{k-i}+o(|x|^{k-i}),
& \text{if } x\in U_0(\gamma,k), \\[0.5ex]
D^i h(x)
=
i-2+C|x-1|^{\gamma-i}+o(|x-1|^{\gamma-i}),
& \text{if } x\in U_{1_-}(\gamma,k), \\[0.5ex]
D^i h(x)
=
2-i+C|x+1|^{\gamma-i}+o(|x+1|^{\gamma-i}),
& \text{if } x\in U_{-1_+}(\gamma,k).
\end{cases}
\end{equation}
The constant \(C\) may depend on the parameters \(\gamma\) and \(k\). We assume throughout that the error terms in \eqref{eqn_2} and \eqref{eqn_3} are uniformly controlled over the parameter range under consideration. These asymptotics are uniform in the parameter range. Consequently, near the
singular point \(0\),
\(
\left|\frac{h''(x)}{h'(x)}\right|\lesssim |x|^{-1},
\)
while near the indifferent endpoints,
\(
\left|\frac{h''(x)}{h'(x)}\right|
\lesssim |1-x|^{\gamma-2}
\quad\text{or}\quad
|1+x|^{\gamma-2},
\)
with constants uniform over the family. These are the distortion kernels used below.
\noindent
In addition, each map in the family is assumed to satisfy the following structural conditions:
\begin{enumerate}
    \item[(i)] \(|Dh(x)|>1\) for all \(x\in I\setminus\{\pm1,0\}\);
    \item[(ii)] \(h\) is strictly increasing on \(I^-\), strictly decreasing on \(I^+\), and convex on each of the intervals \((-1,0)\) and \((0,1)\);
    \item[(iii)] the two branches of \(h\) are onto \(I\).
\end{enumerate}

\noindent
For
\(
0<k<\min\left\{\frac{1}{\gamma-1},1\right\},
\)
the deterministic map \(h\) admits an absolutely continuous invariant probability measure and has polynomial decay of correlations with rate
\(
n^{-\frac{1-k(\gamma-1)}{k(\gamma-1)}}
\)
\cite{gnsp}.

\medskip

\noindent
In the general random notation introduced above, the parameter space for this family is
\[
\mathcal A
=
[\gamma_1,\gamma_2]\times[k_1,k_2]\times[a_1,a_2]\times[b_1,b_2],
\]
and for
\(
\omega_0=(\gamma(\omega),k(\omega),a(\omega),b(\omega))\in\mathcal A
\)
we write
\(
h_\omega:=h_{\omega_0}.
\)
The associated random cocycle is denoted by
\(
h_\omega^n
=
h_{\sigma^{n-1}\omega}\circ\cdots\circ h_{\sigma\omega}\circ h_\omega\),
\(n\geq1.
\)
The corresponding skew product is denoted by
\[
G(\omega,x)=(\sigma\omega,h_\omega(x)).
\]

\noindent
Unlike the Pikovsky family, the maps \(h_\omega\) do not preserve Lebesgue measure in general. Hence the random Grossmann--Horner theorem will be stated with respect to an equivariant family of absolutely continuous probability measures
\(
\{\mu_\omega\}_{\omega\in\Omega}\),
\( (h_\omega)_*\mu_\omega=\mu_{\sigma\omega}.
\)

The parameter restrictions fixed above include
\(
k_2(\gamma_2-1)<1.
\)
This condition is used to obtain the uniform finiteness of the random towers. The
quenched tail estimate, and hence the correlation rate, is governed by
\(
\zeta := k_2(\gamma_1-1).
\)

Thus \(1/\zeta\) is the return time tail exponent appearing in the correlation
bounds below.

\subsection{Statement of results}

\subsubsection{General setting}
The main results below are formulated in terms of general lower drift
conditions for the random endpoint recursions. These conditions are separated
from the concrete parameter laws because the tower tail estimates depend only
on the quenched drift along the random orbit. The probabilistic structure
underlying these estimates is different in the Pikovsky and Grossmann--Horner
cases, and this difference is reflected in the form of the corresponding
triangular arrays. 

\medskip

\noindent
We first state the assumption for random Pikovsky maps. For \(1\leq k\leq n\), define
\begin{align}\label{def_A_nk}
\begin{split}
A_{n,k}(\omega)
&\coloneqq
(\alpha_1-1)
\left[
\frac{1}{2\alpha(\sigma^{n-k}\omega)}
+
m_k(\sigma^{n-k}\omega)
\right]
\left[
\frac{c_k(\alpha_1)}
{k^{1/(\alpha_1-1)}}
\right]^{\alpha(\sigma^{n-k}\omega)-\alpha_1}
\\
&\quad
-\frac{\alpha_1(\alpha_1-1)}{2}
\left[
\frac{1}{2\alpha(\sigma^{n-k}\omega)}
+
m_k(\sigma^{n-k}\omega)
\right]^2
\left[
\frac{c_k(\alpha_2)}
{k^{1/(\alpha_2-1)}}
\right]^{2\alpha(\sigma^{n-k}\omega)-\alpha_1-1}.
\end{split}
\end{align}
Here \(c_k(\alpha_1)\) and \(c_k(\alpha_2)\) are the deterministic endpoint constants
associated with the maps \(g_{\alpha_1}\) and \(g_{\alpha_2}\), respectively. The term
\(m_k(\sigma^{n-k}\omega)\) is defined in~\eqref{m_k}. It is written as a function of
\(\sigma^{n-k}\omega\) since it depends on the corresponding random endpoint sequence and is
measurable with respect to the finite block \(\sigma(\omega_{n-k},\ldots,\omega_n)\).

\begin{description}
\item[\namedlabel{itm:A1}{\textbf{(A)}}]
There exist constants \(q=q(\nu)\geq 0\), \(c_A>0\), \(C_A>0\), \(u_A>0\), and
\(0<v_A<1\), together with a random variable
\(
N_A:\Omega\to\mathbb N,
\)
such that
\(
\mathbb P\{N_A>n\}\leq C_A e^{-u_A n^{v_A}}
\)
for every \(n\in\mathbb N\), and for \(\mathbb P\)-almost every \(\omega\in\Omega\),
\begin{equation}\label{assumpA1}
\frac{(\log n)^q}{n}
\sum_{k=1}^n A_{n,k}(\omega)
\geq c_A
\end{equation}
for every \(n\geq N_A(\omega)\).
\end{description}

\begin{thm}\label{thm2}
Suppose that the random Pikovsky maps satisfy Assumption~\ref{itm:A1}. Then,
for every \(\delta>0\), there exist a full measure set
\(\Omega_0\subset\Omega\) and a random variable
\(
C_\omega:\Omega_0\to\mathbb R_+
\)
such that the following holds. For every fibre observable
\(
\varphi=(\varphi_\omega)_{\omega\in\Omega}\),
\( \psi=(\psi_\omega)_{\omega\in\Omega},
\)
with
\[
\sup_{\omega\in\Omega}\|\varphi_\omega\|_{L^\infty(\mathbb T)}<\infty,
\qquad
\sup_{\omega\in\Omega}\|\psi_\omega\|_{\mathcal C^\eta(\mathbb T)}<\infty,
\]
there exists \(C_{\varphi,\psi}>0\), independent of \(\omega\) and \(n\), such that for every
\(\omega\in\Omega_0\) and every \(n\geq1\),

\begin{equation}\label{F-Cor1}
\left|\int_{\mathbb T} (\varphi_{\sigma^n\omega}\circ g_\omega^n)\psi_\omega\,dm
-
\int_{\mathbb T} \varphi_{\sigma^n\omega}dm
\int_{\mathbb T} \psi_\omega dm
    \right|
\leq
C_\omega C_{\varphi,\psi}
n^{-\left(\frac{1}{\alpha_1-1}-1-\delta\right)};
\end{equation}

\begin{equation}\label{P-Cor1}
\left|\int_{\mathbb T} (\varphi_\omega\circ g_{\sigma^{-n}\omega}^n)\psi_{\sigma^{-n}\omega}\,dm
-
\int_{\mathbb T} \varphi_\omega\,dm
\int_{\mathbb T} \psi_{\sigma^{-n}\omega}\,dm
\right|
\leq
C_\omega C_{\varphi,\psi}
n^{-\left(\frac{1}{\alpha_1-1}-1-\delta\right)}.
\end{equation}

Moreover, there exist constants \(C>0\), \(u'>0\), and \(0<v'<1\) such that
\(
\mathbb P\{C_\omega>n\}\leq Ce^{-u'n^{v'}}
\)
for every \(n\in\mathbb N\).
\end{thm}
\noindent
In particular, deterministic observables \(\varphi\in L^\infty(\mathbb{T})\) and
\(\psi\in\mathcal C^\eta(\mathbb{T})\) are included by identifying them with constant fibre families,
\(
\varphi_\omega\equiv\varphi\),
\( \psi_\omega\equiv\psi.
\)

The formulation of the Pikovsky result is simpler than the corresponding Grossmann--Horner result because each map \(g_\omega\) preserves Lebesgue measure \(m\). Hence the equivariant family of sample measures is simply \(\mu_\omega=m\) for all \(\omega\in\Omega\).

\medskip

\noindent
We now state the analogous assumption for the random Grossmann--Horner maps. For
\(1\leq k\leq n\), define
\begin{align}\label{def_B_nk}
\begin{split}
B_{n,k}(\omega)
&\coloneqq
\left[
a(\sigma^{n-k}\omega)
+
u_{\gamma(\sigma^{n-k}\omega)}
\bigl(x_k^-(\sigma^{n-k}\omega)\bigr)
\right]
\left[
\frac{c_k(\gamma_1)}
{k^{1/(\gamma_1-1)}}
\right]^{\gamma(\sigma^{n-k}\omega)-\gamma_1}
\\
&\quad
-\frac{\gamma_1-1}{2}
\left[
a(\sigma^{n-k}\omega)
+
u_{\gamma(\sigma^{n-k}\omega)}
\bigl(x_k^-(\sigma^{n-k}\omega)\bigr)
\right]^2
\left[
\frac{c_k(\gamma_2)}
{k^{1/(\gamma_2-1)}}
\right]^{2\gamma(\sigma^{n-k}\omega)-\gamma_1-1}.
\end{split}
\end{align}
The endpoint \(x_k^-(\sigma^{n-k}\omega)\) is the one appearing in the negative endpoint
recursion for the Grossmann--Horner family. If the positive endpoint recursion is used
instead, the definition of \(B_{n,k}\) is replaced by the corresponding expression
obtained from that recursion.

\begin{description}
\item[\namedlabel{itm:B1}{\textbf{(B)}}]
There exist constants \(q=q(\nu)\geq 0\), \(c_B>0\), \(C_B>0\), \(u_B>0\), and
\(0<v_B<1\), together with a random variable
\(
N_B:\Omega\to\mathbb N,
\)
such that
\(
\mathbb P\{N_B>n\}\leq C_B e^{-u_B n^{v_B}}
\)
for every \(n\in\mathbb N\), and for \(\mathbb P\)-almost every \(\omega\in\Omega\),
\begin{equation}\label{assumpB1}
\frac{(\log n)^q}{n}
\sum_{k=1}^n B_{n,k}(\omega)
\geq c_B
\end{equation}
for every \(n\geq N_B(\omega)\).
\end{description}

\begin{thm}\label{thm3}
Suppose that the random Grossmann--Horner maps satisfy Assumption~\ref{itm:B1}.
Then there exists a measurable family of absolutely continuous probability measures
\(
\{\mu_\omega\}_{\omega\in\Omega}
\)
on \(I\) such that
\(
(h_\omega)_*\mu_\omega=\mu_{\sigma\omega}
\)
for \(\mathbb P\)-almost every \(\omega\in\Omega\).

Consequently, the measure \(\boldsymbol{\mu}\) on \(\Omega\times I\) defined by
\[
\boldsymbol{\mu}(A)
=
\int_\Omega \mu_\omega(A_\omega)\,d\mathbb P(\omega),
\qquad
A_\omega=\{x\in I:(\omega,x)\in A\},
\]
is invariant under the skew product
\(
G(\omega,x)=(\sigma\omega,h_\omega(x)).
\)

Moreover, for every \(\delta>0\), there exist a full measure set
\(\Omega_0\subset\Omega\) and a random variable
\(
C_\omega:\Omega_0\to\mathbb R_+
\)
such that the following holds. For every fibre observable
\(
\varphi=(\varphi_\omega)_{\omega\in\Omega},
\psi=(\psi_\omega)_{\omega\in\Omega},
\)
with
\(
\sup_{\omega\in\Omega}\|\varphi_\omega\|_{L^\infty(I)}<\infty,
\sup_{\omega\in\Omega}\|\psi_\omega\|_{\mathcal C^\eta(I)}<\infty,
\)
there exists a constant \(C_{\varphi,\psi}>0\), independent of \(n\) and \(\omega\),
such that for every \(\omega\in\Omega_0\) and every \(n\geq1\),

    \begin{equation}\label{F-Cor2}
    \left|\int_I (\varphi_{\sigma^n\omega}\circ h_\omega^n)\psi_\omega\,d\mu_\omega
-
\int_I \varphi_{\sigma^n\omega}\,d\mu_{\sigma^n\omega}
\int_I \psi_\omega\,d\mu_\omega
    \right|
     \leq
    C_\omega C_{\varphi,\psi}
    n^{-\left(\frac{1}{\zeta}-1-\delta\right)};
    \end{equation}

    \begin{equation}\label{P-Cor2}
   \left|\int_I (\varphi_\omega\circ h_{\sigma^{-n}\omega}^n)\psi_{\sigma^{-n}\omega}\,d\mu_{\sigma^{-n}\omega}
-
\int_I \varphi_\omega\,d\mu_\omega
\int_I \psi_{\sigma^{-n}\omega}\,d\mu_{\sigma^{-n}\omega}
\right|
    \leq
     C_\omega C_{\varphi,\psi}
     n^{-\left(\frac{1}{\zeta}-1-\delta\right)} .
    \end{equation}

Moreover, there exist constants \(C>0\), \(u'>0\), and \(0<v'<1\) such that
\(
\mathbb P\{C_\omega>n\}\leq C e^{-u'n^{v'}}
\)
for every \(n\in\mathbb N\).
\end{thm}
\medskip

\noindent
In both cases, the exponent \((\cdot-1)\) in the correlation rate reflects the summation over return-time tails in the inducing scheme. For the Grossmann--Horner system, the quantity \(1/\zeta\) corresponds to the tail exponent of the return time function.

\noindent

\subsubsection{Natural Parameter Law}
The above drift assumptions are the main probabilistic input in the tower
construction. We verify them in Section~\ref{sec:DDUD} for the concrete choices of
parameter law below. Since the endpoint arrays generated by the random cocycle are not independent,
the verification uses lower bounds by one coordinate comparison arrays. These
comparison arrays are independent for each fixed row and satisfy the required
concentration estimates, resulting in the quenched lower drift bounds with
stretched exponential random cutoffs appearing in Assumptions~\ref{itm:A1}
and~\ref{itm:B1}.

For Pikovsky maps: Let \(\beta_1<\beta_2\).

\begin{description}
  \item[\namedlabel{itm:DD}{DD}]
  We say that \(\nu\) has a \emph{two point distribution} on \([\beta_1,\beta_2]\) if
  \[
  \nu=p_1\delta_{\beta_1}+p_2\delta_{\beta_2},
  \qquad
  p_1,p_2>0,\quad p_1+p_2=1.
  \]

  \item[\namedlabel{itm:UD}{UD}]
  We say that \(\nu\) has the \emph{uniform distribution} on \([\beta_1,\beta_2]\) if
  \[
  d\nu(t)=\frac{1}{\beta_2-\beta_1}\,dt,
  \qquad t\in[\beta_1,\beta_2].
  \]
\end{description}

We first verify the drift condition for the discrete and uniform Pikovsky parameter laws.
\begin{prop}
\label{prop:DDUD-imply-A}
Suppose that the Pikovsky parameter law is either
\[
\nu=p_1\delta_{\alpha_1}+p_2\delta_{\alpha_2},
\qquad p_1,p_2>0,\quad p_1+p_2=1,
\]
or the normalized Lebesgue measure on \([\alpha_1,\alpha_2]\). Then
Assumption~\ref{itm:A1} holds. More precisely, \(q=0\) in the discrete case and
\(q=1\) in the uniform case.
\end{prop}

\begin{cor}\label{thm1}
Suppose that the parameter law \(\nu\) satisfies
\ref{itm:DD} or \ref{itm:UD}. Then, for every \(\delta>0\), the conclusions of
Theorem~\ref{thm2} hold. In particular, there exist a full measure set
\(\Omega_0\subset\Omega\) and a random variable
\(
C_\omega:\Omega_0\to\mathbb R_+
\)
such that the future and past correlation bounds
\eqref{F-Cor1} and \eqref{P-Cor1} hold for every
\(\omega\in\Omega_0\) and every \(n\geq1\).

Moreover, there exist constants \(C,u'>0\) and \(0<v'<1\) such that
\(
\mathbb P\{C_\omega>n\}\leq Ce^{-u'n^{v'}}
\)
for every \(n\in\mathbb N\).
\end{cor}

Since
\[
\mathcal A=[\gamma_1,\gamma_2]\times[k_1,k_2]\times[a_1,a_2]\times[b_1,b_2]
\]
for the Grossmann--Horner family, we use the following parameter laws on
\(\mathcal A\):

\begin{description}
 \item[\namedlabel{itm:GHDD}{GH-DD}]
 \[
 \nu=p_1\delta_{\theta^{(1)}}+p_2\delta_{\theta^{(2)}},
 \qquad p_1,p_2>0,\quad p_1+p_2=1,
 \]
 where
 \[
 \theta^{(i)}=(\gamma^{(i)},k^{(i)},a^{(i)},b^{(i)})\in\mathcal A,
 \]
 and at least one atom satisfies \(\gamma^{(i)}=\gamma_1\).

 \item[\namedlabel{itm:GHUD}{GH-UD}]
 \(\nu\) is the normalized Lebesgue measure on \(\mathcal A\), namely
 \[
 d\nu(\gamma,k,a,b)
 =
 \frac{d\gamma\,dk\,da\,db}
 {(\gamma_2-\gamma_1)(k_2-k_1)(a_2-a_1)(b_2-b_1)}.
 \]
\end{description}
We next verify the corresponding drift condition for the discrete and uniform
Grossmann--Horner parameter laws.
\begin{prop}
\label{prop:GHDDUD-imply-B}
Suppose that the Grossmann--Horner parameter law is either ~\ref{itm:GHDD} or ~\ref{itm:GHUD} on
\[
[\gamma_1,\gamma_2]\times[k_1,k_2]\times[a_1,a_2]\times[b_1,b_2].
\]
Then Assumption~\ref{itm:B1} holds. More precisely, \(q=0\) in the discrete case and
\(q=1\) in the uniform case.
\end{prop}

\begin{cor}\label{thm:GH}
Suppose that the parameter law \(\nu\) satisfies
\ref{itm:GHDD} or \ref{itm:GHUD}. Then, for every \(\delta>0\), the conclusions
of Theorem~\ref{thm3} hold. In particular, there exists a measurable family
of absolutely continuous probability measures
\(
\{\mu_\omega\}_{\omega\in\Omega}
\)
on \(I\) such that
\(
(h_\omega)_*\mu_\omega=\mu_{\sigma\omega}
\)
for \(\mathbb P\)-almost every \(\omega\in\Omega\). Moreover, the measure
\(
d\boldsymbol{\mu}(\omega,x)=d\mu_\omega(x)\,d\mathbb P(\omega)
\)
is \(G\)-invariant, and there exist a full measure set
\(\Omega_0\subset\Omega\) and a random variable
\(
C_\omega:\Omega_0\to\mathbb R_+
\)
such that the future and past correlation bounds
\eqref{F-Cor2} and \eqref{P-Cor2} hold for every
\(\omega\in\Omega_0\) and every \(n\geq1\).

Moreover, there exist constants \(C,u'>0\) and \(0<v'<1\) such that
\(
\mathbb P\{C_\omega>n\}\leq Ce^{-u'n^{v'}}
\)
for every \(n\in\mathbb N\).
\end{cor}

\section{Overview of the Proof}\label{sec:overview}

The proof of the main results is based on the construction of random
Young towers. The two families of maps considered in this paper have
different geometric features, but the inducing procedure follows the
same general principle. In both cases one constructs, for
$\mathbb P$-almost every $\omega\in\Omega$, a random inducing scheme
over a suitable base $\Lambda_\omega$.

Let $f_\omega$ denote either the random Pikovsky map $g_\omega$ or the
random Grossmann--Horner map $h_\omega$. The inducing scheme consists of
a partition $\mathcal P_\omega$ of $\Lambda_\omega$ and a return time
function
\(
R_\omega:\Lambda_\omega\to\mathbb N
\)
such that the induced map
\[
f_\omega^{R_\omega}:\Lambda_\omega\to
\Lambda_{\sigma^{R_\omega}\omega}
\]
is a full branch expanding map on every element of
$\mathcal P_\omega$.

The associated random tower is
\[
\Delta_\omega
=
\{(x,\ell):x\in\Lambda_{\sigma^{-\ell}\omega},
\ 0\leq \ell<R_{\sigma^{-\ell}\omega}(x)\},
\]
with tower map
\(
F_\omega:\Delta_\omega\to\Delta_{\sigma\omega}.
\)
The tower projection
\(
\pi_\omega:\Delta_\omega\to X
\)
satisfies the semiconjugacy relation
\(
\pi_{\sigma\omega}\circ F_\omega
=
f_\omega\circ\pi_\omega
\) where \(X\) denotes the corresponding phase space.

The difference between the two families appears in the construction of
the inducing bases and in the estimates for the return time function.
For the Pikovsky family, the inducing scheme is determined by the
excursions between the two sides of the singular point on
$\mathbb T=I/\{-1\sim1\}$. For the Grossmann--Horner family, the
excursions occur between the singular branch at the origin and the two
neutral fixed points at the endpoints of $[-1,1]$. In both cases, the
main estimates reduce to controlling the random endpoint recursions
generated by the cocycle.

We use the random tower framework of  \cite{BahBosRuz19}. Let
$\bar m_\omega$ denote the reference measure on $\Delta_\omega$.
The tower satisfies the following properties.

\begin{enumerate}
\item[(P1)] \textbf{Markov property.}
For every $\Lambda_j(\omega)\in\mathcal P_\omega$,
\[
F_\omega^{R_\omega}|_{\Lambda_j(\omega)}
:
\Lambda_j(\omega)\to
\Lambda_{\sigma^{R_\omega}\omega}
\]
is a bijection modulo zero.

\item[(P2)] \textbf{Bounded distortion.}
There exist constants $D>0$ and $0<\rho<1$ such that for all
$x,y\in\Lambda_j(\omega)$,
\[
\left|
\frac{JF_\omega^{R_\omega}(x)}
     {JF_\omega^{R_\omega}(y)}
-1
\right|
\leq
D\rho^{s(F_\omega^{R_\omega}x,F_\omega^{R_\omega}y)} .
\]

\item[(P3)] \textbf{Generating property.}
The partition $\mathcal P_\omega$ is generating for the induced system.

\item[(P4)] \textbf{Quenched return time tails.}
There exist constants $C>0$, $a>1$, $b\geq0$, $u,v>0$ and a
random variable $n_1(\omega)$ with stretched exponential tails such
that
\[
\bar m_\omega\{R_\omega>n\}
\leq
C\frac{(\log n)^b}{n^a}
\]
for all $n\geq n_1(\omega)$.

\item[(P5)] \textbf{Aperiodicity.}
There exist return times
$t_1,\dots,t_N$ with
\[
\gcd(t_1,\dots,t_N)=1
\]
and constants $\epsilon_i>0$ such that
\(
\bar m_\omega\{R_\omega=t_i\}>\epsilon_i
\)
for $\mathbb P$-almost every $\omega$.

\item[(P6)] \textbf{Finiteness.}
There exists $M>0$ such that
\(
\bar m_\omega(\Delta_\omega)\leq M
\)
for $\mathbb P$-almost every $\omega$.

\item[(P7)] \textbf{Annealed return time tails.}
There exist constants $C>0$, $\hat b\geq0$ and $a>1$ such that
\[
\int_\Omega
\bar m_\omega\{R_\omega=n\}\,d\mathbb P(\omega)
\leq
C\frac{(\log n)^{\hat b}}{n^{a+1}} .
\]
\end{enumerate}

The geometric properties \((P1)\), \((P3)\), and \((P5)\) follow from
the explicit structure of the induced branches. To verify \((P2)\), we
decompose each return branch into its singular and neutral components,
together with a uniformly controlled transition component when it is
present. In distance coordinates, the singular contribution is
controlled by the initial logarithmic relative width, while the
one step neutral contributions telescope along the excursion. This
gives a Euclidean distortion estimate in terms of the length of the
final image interval. Uniform expansion of the induced branches then
converts this estimate into the separation time bound required in
\((P2)\).

The return time estimates are obtained separately. The endpoint drift assumptions give
the quenched endpoint bounds and return time tails needed for \((P4)\) and \((P7)\).
The uniform finiteness condition \((P6)\) is proved independently, using deterministic
worst parameter summability estimates.

Although the endpoint recursions play the same role in both families,
their raw triangular arrays contain endpoint terms generated by the
cocycle and are not independent in the summation variable. In
Section~\ref{sec:DDUD}, we verify the drift assumptions for the
discrete and uniform laws by bounding the raw arrays from below by
one coordinate comparison arrays. For each fixed \(n\), the comparison
variables depend only on the distinct coordinates
\(\omega_{n-k}\), and hence are independent and uniformly bounded.
Hoeffding's inequality gives exponentially small probabilities for
the corresponding row wise lower drift failures. Summing these
probabilities yields random cutoffs with stretched exponential tails.

The proof is organized according to the following logical structure. First, we
prove the quenched mixing results under the abstract endpoint drift assumptions
\ref{itm:A1} and \ref{itm:B1}. These assumptions provide the lower drift
estimates needed for the endpoint recursions and imply the required tower tail
bounds. The corresponding conditional results are stated in
Theorems~\ref{thm2} and~\ref{thm3}.

The concrete parameter laws are treated afterwards. In Section~\ref{sec:DDUD},
we verify that
\[
\ref{itm:DD}/\ref{itm:UD}\Longrightarrow\ref{itm:A1},
\qquad
\ref{itm:GHDD}/\ref{itm:GHUD}\Longrightarrow\ref{itm:B1}.
\]
Combining these verifications with the conditional theorems gives the stated
discrete and uniform distribution results.

At the level of the individual families, the verification is carried out as
follows. Sections~\ref{sec:pikovsky} and
\ref{secthm3} construct the corresponding inducing schemes and verify the
geometric tower properties for Pikovsky and Grossmann--Horner maps respectively. Under Assumptions~\ref{itm:A1} and
\ref{itm:B1}, the endpoint estimates and return time tail bounds needed for
(P4), (P6), and (P7) are established in these sections. The quenched decay of
correlations in Theorems~\ref{thm2} and~\ref{thm3} then follows from the random
tower results of \cite{BahBosRuz19}. Finally, Section~\ref{sec:DDUD} verifies
the drift assumptions for the discrete and uniform parameter laws.

\section{Random Pikovsky Maps}\label{sec:pikovsky}
In this section we verify the tower properties for the random Pikovsky
family introduced in Section~\ref{sec:overview}. The main ingredients are
the explicit construction of the inducing scheme and the estimates on the
random endpoint sequences. These estimates imply bounded distortion (P2),
quenched return time tails (P4), finiteness (P6), and annealed return time
asymptotics (P7).
\subsection{Inducing construction}\label{topcons1}

For \(n\geq0\), define the random endpoint sequences by
\(
x_0^+(\omega)=x_0^-(\omega)=0,
\)
and
\[
x_{n+1}^+(\omega)
=
(g_\omega|_{I^+})^{-1}
\bigl(x_n^+(\sigma\omega)\bigr),
\qquad
x_{n+1}^-(\omega)
=
(g_\omega|_{I^-})^{-1}
\bigl(x_n^-(\sigma\omega)\bigr).
\]
Define the corresponding intervals
\[
\Delta_n^+(\omega)
=
(x_n^+(\omega),x_{n+1}^+(\omega)),
\qquad
\Delta_n^-(\omega)
=
(x_{n+1}^-(\omega),x_n^-(\omega)).
\]
Then \(\{\Delta_n^+(\omega)\}_{n\geq0}\) and
\(\{\Delta_n^-(\omega)\}_{n\geq0}\) are mod \(0\) partitions of
\(I^+\) and \(I^-\), respectively.

For \(n\geq1\), let
\[
\delta_n^-(\omega)
=
g_\omega^{-1}
\bigl(\Delta_{n-1}^+(\sigma\omega)\bigr)
\cap
\Delta_0^-(\omega),
\quand
\delta_n^+(\omega)
=
g_\omega^{-1}
\bigl(\Delta_{n-1}^-(\sigma\omega)\bigr)
\cap
\Delta_0^+(\omega).
\]
These form mod \(0\) partitions of
\(\Delta_0^-(\omega)\) and \(\Delta_0^+(\omega)\), respectively. Moreover,
\[
g_\omega^n:
\delta_n^-(\omega)\to
\Delta_0^+(\sigma^n\omega),
\qquad
g_\omega^n:
\delta_n^+(\omega)\to
\Delta_0^-(\sigma^n\omega)
\]
are \(C^2\) diffeomorphisms.

We choose the inducing base
\(
\Lambda_\omega=\Delta_0^-(\omega).
\)
For \(i,j\geq1\), define the return partition elements
\[
\delta_{ij}^-(\omega)
=
\delta_i^-(\omega)
\cap
g_\omega^{-i}
\bigl(\delta_j^+(\sigma^i\omega)\bigr).
\]
Then
\(
\mathcal P_\omega
=
\{\delta_{ij}^-(\omega):i,j\geq1\}
\)
is a mod \(0\) partition of \(\Lambda_\omega\), and each branch
\[
g_\omega^{i+j}:
\delta_{ij}^-(\omega)
\longrightarrow
\Lambda_{\sigma^{i+j}\omega}
\]
is a full-branch \(C^2\) diffeomorphism. We define the return time by
\(
R_\omega|_{\delta_{ij}^-(\omega)}=i+j.
\)
The induced map is therefore
\[
\widehat g_\omega
=
g_\omega^{R_\omega}:
\Lambda_\omega
\longrightarrow
\Lambda_{\sigma^{R_\omega}\omega}.
\]

The corresponding tower projection is given by
\(
\pi_\omega(x,\ell)
=
g_{\sigma^{-\ell}\omega}^{\ell}(x),
\)
and satisfies
\(
\pi_{\sigma\omega}\circ F_\omega
=
g_\omega\circ\pi_\omega.
\)

\begin{prop}\label{prop:pikovsky-basic}
The random Pikovsky tower satisfies properties (P1), (P3), and (P5).
\end{prop}

\begin{proof}
By construction,
\[
g_\omega^{R_\omega}|_{\delta_{ij}^-(\omega)}
:
\delta_{ij}^-(\omega)
\to
\Lambda_{\sigma^{R_\omega}\omega}
\]
is a bijection modulo zero, which gives (P1).

The generating property (P3) follows from the uniform expansion of the
induced branches. Indeed, every element of $\mathcal P_\omega$ is mapped
diffeomorphically onto the base and the inverse branches contract
uniformly.

Finally, the return times $2$ and $3$ occur through the elements
$\delta^-_{11}(\omega)$ and $\delta^-_{12}(\omega)$, respectively. By
continuity of the endpoint maps and compactness of the parameter range,
their Lebesgue measures are uniformly bounded below. Since
$\gcd(2,3)=1$, property (P5) follows.
\end{proof}

\subsection{Endpoint estimates}\label{endpoint-estimates-pik}
We verify (P2), (P4), (P6), and (P7) using estimates on the endpoint
sequences. We first introduce the notation
\[
\Delta_n^+(\omega)=(x_n^+(\omega),x_{n+1}^+(\omega)),
\qquad
\Delta_n^-(\omega)=(x_{n+1}^-(\omega),x_n^-(\omega)).
\]
For \(n\geq0\), define
\[
y_n^-(\omega)
=
(g_\omega|_{\Delta_0^-(\omega)})^{-1}
(x_n^+(\sigma\omega)),
\qquad
y_n^+(\omega)
=
(g_\omega|_{\Delta_0^+(\omega)})^{-1}
(x_n^-(\sigma\omega)).
\]
Then
\[
\delta_n^-(\omega)=(y_{n-1}^-(\omega),y_n^-(\omega)),
\qquad
\delta_n^+(\omega)=(y_n^+(\omega),y_{n-1}^+(\omega)).
\]

If \(\omega=(\alpha,\alpha,\ldots)\), these sequences coincide with the
deterministic endpoints of \(g_\alpha\). We use the following estimates from
\cite{gnsp}.

\begin{lem}[Lemma 1 of \cite{gnsp}]\label{lem:eqn3}
For the deterministic map \(g_\alpha\),
\(
x_{n+1}^{\pm}
=
x_n^\pm
\pm
\frac{1}{2\alpha}(1\mp x_n^\pm)^\alpha ,
\)
and
\[
1-x_n^+
\sim
\left(\frac{2\alpha}{\alpha-1}\right)^{-1/(\alpha-1)}
n^{-1/(\alpha-1)},
\quand
x_n^-+1
\sim
\left(\frac{2\alpha}{\alpha-1}\right)^{-1/(\alpha-1)}
n^{-1/(\alpha-1)} .
\]
Moreover,
\[
m(\Delta_n^\pm)
\sim
\frac{1}{2\alpha}
\left(\frac{2\alpha}{\alpha-1}\right)^{-\alpha/(\alpha-1)}
n^{-\alpha/(\alpha-1)},
\quand
|y_n^\pm|
\sim
\left(\frac{2\alpha}{\alpha-1}\right)^{-\alpha/(\alpha-1)}
n^{-\alpha/(\alpha-1)} .
\]
\end{lem}

Near \(1^-\), \(g_\alpha(x)=1-(1-x)-(1-x)^\alpha
\left(\frac1{2\alpha}+u_\alpha(x)\right)\)
\begin{align}\label{lem:eqn3'}
\begin{split}
g_\alpha'(x)
=
1+(1-x)^{\alpha-1}
\left(\frac12+v_\alpha(x)\right), \quand
g_\alpha''(x)
=
(1-x)^{\alpha-2}
\left(-\frac{\alpha-1}{2}+w_\alpha(x)\right),
\end{split}
\end{align}
where
\(
u_\alpha(x),v_\alpha(x),w_\alpha(x)\to0
\quad (x\to1^-).
\)
Moreover,
\[
u_\alpha(x)
=
\frac1{2\alpha}
\left[
\left(\frac{1-g_\alpha(x)}{1-x}\right)^\alpha-1
\right],
\]
and hence
\[
\sup_{\alpha\in[\alpha_1,\alpha_2]}
\sup_{x\in[1/2,1]}
(|u_\alpha(x)|+|v_\alpha(x)|+|w_\alpha(x)|)
\leq M .
\]
The estimates near \(-1\) follow by symmetry.

For later use, write
\(
m_k(\omega)
=
u_{\alpha(\omega)}(x_k^+(\omega)).
\)
Thus
\begin{equation}   
\label{m_k}
   m_k(\sigma^{n-k}\omega)
=
u_{\alpha(\sigma^{n-k}\omega)}
(x_k^+(\sigma^{n-k}\omega)). 
\end{equation}

By monotonicity of the inverse branches,
\begin{equation}\label{ineq_1}
1\mp x_n^\pm(\alpha_1)
\leq
1\mp x_n^\pm(\omega)
\leq
1\mp x_n^\pm(\alpha_2).
\end{equation}

\begin{lem}\label{lem:y-delta-prelim}
For \(n\geq1\),
\[
|y_n^\mp(\omega)|
\asymp
(1\mp x_n^\pm(\sigma\omega))^{\alpha(\omega)},
\]
and
\[
m(\delta_n^\mp(\omega))
\asymp
(1\mp x_{n-1}^\pm(\sigma\omega))^{\alpha(\omega)}
-
(1\mp x_n^\pm(\sigma\omega))^{\alpha(\omega)} .
\]
Consequently,
\[
|y_n^\mp(\omega)|
\lesssim
(1\mp x_n^\pm(\alpha_2))^{\alpha(\omega)},
\]
and
\[
m(\delta_n^\mp(\omega))
\lesssim
(1\mp x_{n-1}^\pm(\alpha_2))^{\alpha(\omega)}
-
(1\mp x_n^\pm(\alpha_2))^{\alpha(\omega)} .
\]
\end{lem}

\begin{proof}
The first estimates follow directly from the definition of the inverse
branches. For example,
\[
y_n^+(\omega)
=
\frac1{2\alpha(\omega)}
(1+x_n^-(\sigma\omega))^{\alpha(\omega)}.
\]
Since \(\alpha(\omega)\in[\alpha_1,\alpha_2]\), the multiplicative factor is
uniformly bounded above and below. The estimates for
\(m(\delta_n^\pm)\) follow by taking differences. The comparison with the
deterministic endpoints follows from \eqref{ineq_1}.
\end{proof}

\begin{lem}\label{lem:low-xn}
Assume \ref{itm:A1}. Then there exist constants \(C>0\), \(u>0\), \(v>0\)
and a random variable \(n_1:\Omega\to\mathbb N\) such that
\[
1\mp x_n^\pm(\omega)
\leq
C\left(\frac{(\log n)^q}{n}\right)^{1/(\alpha_1-1)}
\]
for all \(n\geq n_1(\omega)\) and \(\mathbb P\)-a.e. \(\omega\). Moreover,
\(
\mathbb P\{n_1>n\}\leq Ce^{-un^v}
\)
for every \(n\in\mathbb N\).
\end{lem}

\begin{proof}
We prove the estimate on the positive side; the negative side follows by symmetry.
Using \eqref{lem:eqn3'} and the inequality
\[
(1+y)^{-\beta}
\leq
1-\beta y+\frac{\beta(1+\beta)}{2}y^2,
\qquad y\geq0,\ \beta>0,
\]
we obtain
\begin{align*}
&[1-x_n^+(\omega)]^{-(\alpha_1-1)}
-
[1-x_{n-1}^+(\sigma\omega)]^{-(\alpha_1-1)}
\\
&\quad\geq
(\alpha_1-1)f_n(\omega)
[1-x_n^+(\omega)]^{\alpha(\omega)-\alpha_1}
-
\frac{\alpha_1(\alpha_1-1)}{2}
f_n(\omega)^2
[1-x_n^+(\omega)]^{2\alpha(\omega)-\alpha_1-1},
\end{align*}
where
\[
f_n(\omega)
=
\frac{1}{2\alpha(\omega)}
+
u_{\alpha(\omega)}(x_n^+(\omega)).
\]
Iterating along the shifted sequence \(x_k^+(\sigma^{n-k}\omega)\), and using the
deterministic comparisons with the endpoints for \(\alpha_1\) and \(\alpha_2\), gives
\[
[1-x_n^+(\omega)]^{-(\alpha_1-1)}
\geq
\sum_{k=1}^n A_{n,k}(\omega).
\]
By Assumption~\ref{itm:A1}, there are \(c_A>0\) and \(N_A(\omega)\) with
stretched exponential tail such that
\[
\frac{(\log n)^q}{n}
\sum_{k=1}^n A_{n,k}(\omega)
\geq c_A
\]
for all \(n\geq N_A(\omega)\). Hence, for such \(n\),
\[
\frac{(\log n)^q}{n[1-x_n^+(\omega)]^{\alpha_1-1}}
\geq c_A,
\]
and therefore
\[
1-x_n^+(\omega)
\leq
c_A^{-1/(\alpha_1-1)}
\left(\frac{(\log n)^q}{n}\right)^{1/(\alpha_1-1)}.
\]
Taking \(n_1=N_A\), after enlarging constants, proves the claim.
\end{proof}
\begin{cor}\label{cor:y_n}
Assume \ref{itm:A1}. There exist constants \(C>0\), \(u>0\), \(v>0\) and a
random variable \(n_1:\Omega\to\mathbb N\) such that
\[
|y_n^\pm(\omega)|
\leq
C\left(\frac{(\log n)^q}{n}\right)^{\alpha_1/(\alpha_1-1)}
\]
for all \(n\geq n_1(\omega)\) and \(\mathbb P\)-a.e. \(\omega\). Moreover,
\(
\mathbb P\{n_1>n\}\leq Ce^{-un^v}
\)
for every \(n\in\mathbb N\).
\end{cor}

\begin{proof}
We prove the estimate for \(y_n^-\); the other side is identical. Since
\(
g_\omega(y_n^-(\omega))=x_n^+(\sigma\omega),
\)
the defining equation for \(g_\omega\) gives
\[
|y_n^-(\omega)|
=
\frac{1}{2\alpha(\omega)}
\bigl(1-x_n^+(\sigma\omega)\bigr)^{\alpha(\omega)}.
\]
By Lemma~\ref{lem:low-xn}, for \(n\geq n_1(\sigma\omega)\),
\[
1-x_n^+(\sigma\omega)
\leq
C\left(\frac{(\log n)^q}{n}\right)^{1/(\alpha_1-1)}.
\]
Since \(\alpha(\omega)\geq\alpha_1\), this implies
\[
|y_n^-(\omega)|
\leq
C\left(\frac{(\log n)^q}{n}\right)^{\alpha_1/(\alpha_1-1)}.
\]
Replacing \(n_1(\omega)\) by \(n_1(\omega)\vee n_1(\sigma\omega)\) preserves the
stretched exponential tail after changing the constants. This completes the proof.
\end{proof}
\begin{cor}\label{cor:delta_n-and-Delta_n}
There exist constants \(C>0\), \(u>0\), \(v>0\), and a random variable
\(n_1:\Omega\to\mathbb N\) such that, for all \(n\geq n_1(\omega)\),
\[
m(\Delta_n^{\pm}(\omega))
\leq
C\left(\frac{(\log n)^q}{n}\right)^{\frac{1}{\alpha_1-1}},
\qquad
m(\delta_n^{\pm}(\omega))
\leq
C\left(\frac{(\log n)^q}{n}\right)^{\frac{\alpha_1}{\alpha_1-1}}.
\]
Moreover,
\(
\mathbb P\{n_1(\omega)>n\}
\leq
C e^{-un^v}\)
 for all  \(n\in\mathbb N.\)
\end{cor}
\subsection{Distortion and Tails}\label{sec:distNtail}

\subsubsection{Distortion estimates}\label{sec:dist}

In this section, we verify the condition \((P2)\) as defined in Section \ref{sec:overview}. The estimate is first
proved for the induced map
\[
\widehat g_\omega=g_\omega^{R_\omega}:\Lambda_\omega\to
\Lambda_{\sigma^{R_\omega}\omega}.
\]
Since the tower map has unit Jacobian along tower levels, the Jacobian of
\(F_\omega^{R_\omega}\) on the base agrees with the Jacobian of
\(\widehat g_\omega\).

Each induced branch is the composition of two one sided excursions. For
a single excursion, the singular step is estimated in distance
coordinates from the singular point and the corresponding neutral
endpoint. The distortion accumulated during the subsequent neutral
iterates telescopes in endpoint distance coordinates, yielding a bound
in terms of the length of the final image. Applying this estimate to
the two excursions and using uniform expansion gives the Euclidean
distortion bound for the induced branch, and hence condition \((P2)\).

We first control the contribution of the singular region at the origin.
\begin{lem}
\label{lem:Pik-singular-step}
There exists \(C>0\) such that the following holds. For
\(\omega\in\Omega\), let
\[
\Psi_\omega(t)
:=
\bigl(2\alpha(\omega)t\bigr)^{1/\alpha(\omega)},
\qquad
0<t\leq\frac{1}{2\alpha(\omega)}.
\]
This is the expression of the inverse singular branch in the distance
coordinates of \(0\) in the domain and from the corresponding
neutral endpoint in the image. Then, for every
\(
0<u<v\leq 1/{2\alpha(\omega)},
\)
we have
\[
\left|
\log
\frac{D\Psi_\omega(v)}
     {D\Psi_\omega(u)}
\right|
\leq
C
\log
\frac{\Psi_\omega(v)}
     {\Psi_\omega(u)}.
\]
The constant \(C\) is independent of \(\omega\).
\end{lem}

\begin{proof}
Write
\(
\alpha=\alpha(\omega).
\)
Since
\(
\Psi_\omega(t)=(2\alpha t)^{1/\alpha},
\)
we have
\(
D\Psi_\omega(t)
=
\frac{1}{\alpha}
(2\alpha)^{1/\alpha}
t^{1/\alpha-1}.
\)
Therefore,
\[
\begin{aligned}
\left|
\log
\frac{D\Psi_\omega(v)}
     {D\Psi_\omega(u)}
\right|
=
\frac{\alpha-1}{\alpha}
\log\frac vu=
(\alpha-1)
\log
\frac{\Psi_\omega(v)}
     {\Psi_\omega(u)}\leq
(\alpha_2-1)
\log
\frac{\Psi_\omega(v)}
     {\Psi_\omega(u)}.
\end{aligned}
\]
\end{proof}

We now control the contribution of the neutral region; near \(\pm 1\).
\begin{lem}
\label{lem:Pik-neutral-step}
There exists \(C>0\) such that the following holds. For
\(\omega\in\Omega\), let
\(
\Phi_\omega:
[0,1-1/{2\alpha(\omega)}]
\longrightarrow[0,1]
\)
be defined implicitly by
\[
t
=
\Phi_\omega(t)
-
1/{2\alpha(\omega)}
\Phi_\omega(t)^{\alpha(\omega)}.
\]
Thus \(\Phi_\omega\) is the expression of the inverse neutral branch in distance coordinates from the corresponding neutral endpoint. Then, for every
\(
0<u<v\leq1-1/{2\alpha(\omega)},
\)
we have
\[
\left|
\log
\frac{D\Phi_\omega(v)}
     {D\Phi_\omega(u)}
\right|
\leq
C\left(
\log
\frac{\Phi_\omega(v)}
     {\Phi_\omega(u)}
-
\log\frac vu
\right).
\]
The constant \(C\) is independent of \(\omega\).
\end{lem}

\begin{proof}
Fix \(\omega\in\Omega\), and write
\(
\alpha=\alpha(\omega), s=\Phi_\omega(t).
\)
The defining relation is
\(
t=s-1/{2\alpha}s^\alpha.
\)
Differentiating gives
\[
D\Phi_\omega(t)
=
\frac{1}
{1-\frac12\Phi_\omega(t)^{\alpha-1}}.
\]
A further differentiation yields
\[
\frac{D^2\Phi_\omega(t)}
     {D\Phi_\omega(t)}
=
\frac{\alpha-1}{2}
\frac{
\Phi_\omega(t)^{\alpha-2}
}{
\left(
1-\frac12\Phi_\omega(t)^{\alpha-1}
\right)^2
}.
\]
On the other hand,
\(t=\Phi_\omega(t)\left(1-\Phi_\omega(t)^{\alpha-1}/{2\alpha}\right)\),
and hence
\[
\frac{d}{dt}
\log\frac{\Phi_\omega(t)}{t}
=
\frac{\alpha-1}{2\alpha}
\frac{
\Phi_\omega(t)^{\alpha-2}
}{
\left(
1-\frac12\Phi_\omega(t)^{\alpha-1}
\right)
\left(
1-\frac{1}{2\alpha}
\Phi_\omega(t)^{\alpha-1}
\right)
}.
\]
Consequently,
\[
\frac{
\displaystyle
\frac{D^2\Phi_\omega(t)}
     {D\Phi_\omega(t)}
}{
\displaystyle
\frac{d}{dt}
\log\frac{\Phi_\omega(t)}{t}
}
=
\alpha
\frac{
1-\frac{1}{2\alpha}
\Phi_\omega(t)^{\alpha-1}
}{
1-\frac12
\Phi_\omega(t)^{\alpha-1}
}.
\]
Since
\(
0\leq\Phi_\omega(t)\leq1,
\)
we have
\(
1-\frac12\Phi_\omega(t)^{\alpha-1}
\geq\frac12\)
and
\(1-\frac{1}{2\alpha}
\Phi_\omega(t)^{\alpha-1}
\leq1.
\)
It follows that
\[
\left|
\frac{D^2\Phi_\omega(t)}
     {D\Phi_\omega(t)}
\right|
\leq
2\alpha_2
\frac{d}{dt}
\log\frac{\Phi_\omega(t)}{t}.
\]
Integrating between \(u\) and \(v\) gives
\[
\begin{aligned}
\left|
\log
\frac{D\Phi_\omega(v)}
     {D\Phi_\omega(u)}
\right|
\leq
\int_u^v
\left|
\frac{D^2\Phi_\omega(t)}
     {D\Phi_\omega(t)}
\right|\,dt\leq
2\alpha_2
\int_u^v
\frac{d}{dt}
\log\frac{\Phi_\omega(t)}{t}\,dt=
2\alpha_2
\left(
\log
\frac{\Phi_\omega(v)}
     {\Phi_\omega(u)}
-
\log\frac vu
\right).
\end{aligned}
\]
\end{proof}

The next estimate controls the logarithmic distortion of one-sided Pikovsky
excursions.
\begin{lem}
\label{lem:Pik-excursion-distortion}
There exists \(C>0\) such that, for every \(\omega\in\Omega\), every
\(n\geq1\), every \(0\leq m<n\), every choice of sign, and all
\(x,y\in\delta_n^\pm(\omega)\),
\[
\left|
\log
\frac{
\left|
Dg_{\sigma^m\omega}^{\,n-m}
\bigl(g_\omega^m(x)\bigr)
\right|
}{
\left|
Dg_{\sigma^m\omega}^{\,n-m}
\bigl(g_\omega^m(y)\bigr)
\right|
}
\right|
\leq
C\left|
g_\omega^n(x)-g_\omega^n(y)
\right|.
\]
In particular, the logarithmic distortion of every one sided
excursion is uniformly bounded.
\end{lem}

\begin{proof}
By odd symmetry, it is enough to consider \(
x,y\in\delta_n^-(\omega).\)
Set
\(
E_r:=g_\omega^r([x,y]),
\quad
0\leq r\leq n.
\)
For \(1\leq r\leq n\),
\(
E_r
\subset
\Delta_{n-r}^+(\sigma^r\omega).
\)
We use distance from the neutral endpoint \(1\), namely
\(d(z):=1-z.\)
Write
\[
d(E_r)=[u_r,v_r],
\quad
0<u_r<v_r,
\quad
\text{ and set } \quad
Q_r:=\log\frac{v_r}{u_r}.
\] 
The distance coordinates are affine changes in coordinates, so the
derivatives of the coordinate maps coincide with the absolute
derivatives of the corresponding branches of \(g_{\sigma^r\omega}\).

For \(1\leq r<n\), the map from \(d(E_r)\) to \(d(E_{r+1})\) is
\(\Phi_{\sigma^r\omega}\). Thus
\(u_{r+1}
=
\Phi_{\sigma^r\omega}(u_r),
\quad
v_{r+1}
=
\Phi_{\sigma^r\omega}(v_r).
\)

Suppose first that \(1\leq m<n\). Since the distance coordinates are
affine isometries, Lemma~\ref{lem:Pik-neutral-step} gives
\begin{equation}
\begin{aligned}
\left|
\log
\frac{
\left|
Dg_{\sigma^m\omega}^{\,n-m}
\bigl(g_\omega^m(x)\bigr)
\right|
}{
\left|
Dg_{\sigma^m\omega}^{\,n-m}
\bigl(g_\omega^m(y)\bigr)
\right|
}
\right|\leq
C\sum_{r=m}^{n-1}(Q_{r+1}-Q_r)=
C(Q_n-Q_m)
\leq
CQ_n.
\end{aligned}
\end{equation}

Now suppose that \(m=0\). Let \([a_0,b_0]\) be the image of \(E_0\)
at distance from the singular point \(0\). The singular coordinate
map satisfies
\(
[u_1,v_1]
=
\Psi_\omega([a_0,b_0]).
\)
Therefore, by Lemma~\ref{lem:Pik-singular-step},
\[
\left|
\log
\frac{|Dg_\omega(x)|}
     {|Dg_\omega(y)|}
\right|
\leq
CQ_1.
\]
The remaining iterates are neutral, and
Lemma~\ref{lem:Pik-neutral-step} gives
\[
\sum_{r=1}^{n-1}
\left|
\log
\frac{
\left|
Dg_{\sigma^r\omega}
\bigl(g_\omega^r(x)\bigr)
\right|
}{
\left|
Dg_{\sigma^r\omega}
\bigl(g_\omega^r(y)\bigr)
\right|
}
\right|
\leq
C(Q_n-Q_1).
\]
The sum is empty when \(n=1\). Hence, in all cases,
\[
\left|
\log
\frac{
\left|
Dg_{\sigma^m\omega}^{\,n-m}
\bigl(g_\omega^m(x)\bigr)
\right|
}{
\left|
Dg_{\sigma^m\omega}^{\,n-m}
\bigl(g_\omega^m(y)\bigr)
\right|
}
\right|
\leq
CQ_n.
\]

Since
\[E_n
\subset
\Delta_0^+(\sigma^n\omega),
\quand
\Delta_0^+(\sigma^n\omega)
=
\left(
0,\frac{1}{2\alpha(\sigma^n\omega)}
\right),
\]
we have
\(
u_n
\geq
d_0
:=
1-\frac{1}{2\alpha_1}
>0.
\)
Consequently,
\[
\begin{aligned}
Q_n
=
\log\frac{v_n}{u_n}=
\log\left(
1+\frac{v_n-u_n}{u_n}
\right)\leq
\frac{v_n-u_n}{u_n}\leq
d_0^{-1}m(E_n).
\end{aligned}
\]
Since
\(
m(E_n)
=
\left|
g_\omega^n(x)-g_\omega^n(y)
\right|,
\)
the desired estimate follows. The positive case is identical by an odd
symmetry.
\end{proof}

We now give Uniform expansion estimates for one sided excursions.
\begin{lem}
\label{lem:Pik-one-sided-expansion}
There exists \(\lambda>1\) such that, for every
\(\omega\in\Omega\), every \(n\geq1\), every choice of sign, and every
\(x\in\delta_n^\pm(\omega)\),
\(
\left|Dg_\omega^n(x)\right|\geq\lambda .
\)
\end{lem}

\begin{proof}
In either singular branch, the distance coordinate map is
\(
\Psi_\omega(t)
=
\bigl(2\alpha(\omega)t\bigr)^{1/\alpha(\omega)} .
\)

Hence
\[
D\Psi_\omega(t)
=
\frac{1}{\alpha(\omega)}
(2\alpha(\omega))^{1/\alpha(\omega)}
t^{1/\alpha(\omega)-1}.
\]
Since this derivative is decreasing in \(t\),

\[
D\Psi_\omega(t)
\geq
D\Psi_\omega
\left(
\frac1{2\alpha(\omega)}
\right)
=\frac{1}{\alpha(\omega)}
(2\alpha(\omega))^{1/\alpha(\omega)}
(2\alpha(\omega))^{-1/\alpha(\omega)+1} = 2,
\]
which proves the desired result.

Every subsequent iterate belongs to a neutral branch, where
\[
D\Phi_{\sigma^r\omega}(t)
=
\frac{1}
{1-\frac12
\Phi_{\sigma^r\omega}(t)^{\alpha(\sigma^r\omega)-1}}
\geq1 .
\]
The result follows from the chain rule.
\end{proof}

We first establish a distortion estimate for the induced branches in terms of the Euclidean distance between their images.
\begin{lem}
\label{lem:Pik-induced-distortion}
There exists \(C>0\) such that, for every \(\omega\in\Omega\), every
\(\delta_{ij}^-(\omega)\in\mathcal P_\omega\), and all
\(x,y\in\delta_{ij}^-(\omega)\),
\[
\left|
\log
\frac{
J\widehat g_\omega(x)
}{
J\widehat g_\omega(y)
}
\right|
\leq
C\left|
\widehat g_\omega(x)-\widehat g_\omega(y)
\right|.
\]
\end{lem}

\begin{proof}
Let
\(
R_\omega|_{\delta_{ij}^-(\omega)}=i+j.
\)
On \(\delta_{ij}^-(\omega), \quad
\widehat g_\omega
=
g_{\sigma^i\omega}^j\circ g_\omega^i.
\)
Set
\(
u=g_\omega^i(x),
\quad
v=g_\omega^i(y).
\)
Then
\(
u,v\in\delta_j^+(\sigma^i\omega).
\)
By the chain rule,
\[
\begin{aligned}
\left|
\log
\frac{
J\widehat g_\omega(x)
}{
J\widehat g_\omega(y)
}
\right|
\leq
\left|
\log
\frac{|Dg_\omega^i(x)|}
     {|Dg_\omega^i(y)|}
\right|\quad+
\left|
\log
\frac{
|Dg_{\sigma^i\omega}^j(u)|
}{
|Dg_{\sigma^i\omega}^j(v)|
}
\right|.
\end{aligned}
\]
By Lemma~\ref{lem:Pik-excursion-distortion}, the second term is bounded
by
\[
C\left|
g_{\sigma^i\omega}^j(u)
-
g_{\sigma^i\omega}^j(v)
\right|
=
C\left|
\widehat g_\omega(x)-\widehat g_\omega(y)
\right|,
\]
and the first term is bounded
by
\(C|u-v|.\)
By Lemma~\ref{lem:Pik-one-sided-expansion}, there exists \(\lambda>1\):
\[
|u-v|
\leq
\frac1\lambda
\left|
g_{\sigma^i\omega}^j(u)
-
g_{\sigma^i\omega}^j(v)
\right|.
\]
Combining these estimates proves the claim.
\end{proof}

We now prove condition (P2).
\begin{prop}\label{distortion}
There exist constants \(D>0\) and \(0<\rho<1\) such that for every
\(\omega\in\Omega\) and every \(x,y\in\delta_{ij}^{-}(\omega)\),
\[
\left|
\frac{JF_\omega^{R_\omega}(x)}
     {JF_\omega^{R_\omega}(y)}
-1
\right|
\leq
D\,\rho^{s(F_\omega^{R_\omega}(x,0),
F_\omega^{R_\omega}(y,0))}.
\]
\end{prop}

\begin{proof}
Let \(R=R_\omega(x)=R_\omega(y)\). Since \(F_\omega^R\) returns the base to
the base and has unit Jacobian along tower levels,
\[
JF_\omega^R(x)=J\widehat g_\omega(x)=|(g_\omega^R)'(x)|.
\]
By Lemma~\ref{lem:Pik-induced-distortion},
\[
\left|
\log
\frac{JF_\omega^R(x)}
     {JF_\omega^R(y)}
\right|
\leq
C
|\widehat g_\omega x-\widehat g_\omega y|.
\]
The induced map is full branch and uniformly expanding. Hence inverse
branches contract exponentially in separation time: there exist
\(C_0>0\) and \(0<\rho<1\) such that
\[
|\widehat g_\omega x-\widehat g_\omega y|
\leq
C_0\rho^{s(F_\omega^R(x,0),F_\omega^R(y,0))}.
\]
Therefore
\[
\left|
\log
\frac{JF_\omega^R(x)}
     {JF_\omega^R(y)}
\right|
\leq
C
\rho^{s(F_\omega^R(x,0),F_\omega^R(y,0))}.
\]
Since the logarithmic distortion is uniformly bounded,
\(|e^t-1|\lesssim |t|\) on the relevant range. Thus
\[
\left|
\frac{JF_\omega^{R_\omega}(x)}
     {JF_\omega^{R_\omega}(y)}
-1
\right|
\leq
D\,\rho^{s(F_\omega^{R_\omega}(x,0),
F_\omega^{R_\omega}(y,0))}.
\]
This is condition \((P2)\).
\end{proof}
\subsubsection{Tail Estimate}\label{sec:tail}

We verify conditions \((P4)\), \((P6)\), and \((P7)\) for the random tower.
The quenched and annealed tail bounds are proved first; the uniform finiteness
condition \((P6)\) is verified separately.

\begin{prop}\label{Tail}
There exist constants \(C>0\), \(u>0\), \(v>0\), a full measure set
\(G\subset\Omega\), and a random variable \(n_2:\Omega\to\mathbb N\) such
that for all \(n\ge n_2(\omega)\),
\[
m(R_\omega>n)
=
\sum_{i+j>n}m(\delta_{ij}^-(\omega))
\lesssim
(\log n)^{\frac{q\alpha_1}{\alpha_1-1}}
\,n^{-\frac1{\alpha_1-1}} .
\]
Moreover,
\(
\mathbb P\{n_2(\omega)>n\}
\le Ce^{-un^v}.
\)

In particular,
\[
\int_\Omega m\{x\in\Lambda_\omega:R_\omega(x)=n\}\,d\mathbb P
\lesssim
(\log n)^{\frac{q\alpha_1}{\alpha_1-1}}
\,n^{-\frac{\alpha_1}{\alpha_1-1}} .
\]
\end{prop}

\begin{lem}\label{lem:cyl}
For each \(i,j\ge1\), bounded distortion implies
\[
m(\delta_{ij}^-(\omega))
\asymp
m(\delta_i^-(\omega))\,
m(\delta_j^+(\sigma^i\omega)).
\]
\end{lem}

\begin{lem}\label{lem:det-tail}
For all \(n\ge n_1(\omega)\),
\[
m(\delta_n^\pm(\omega))
\lesssim
(\log n)^{\frac{q\alpha_1}{\alpha_1-1}}
n^{-\frac{\alpha_1}{\alpha_1-1}} .
\]
\end{lem}

\begin{proof}
This is exactly the second estimate in Corollary~\ref{cor:delta_n-and-Delta_n},
after expanding
\[
\left(\frac{(\log n)^q}{n}\right)^{\alpha_1/(\alpha_1-1)}
=
(\log n)^{q\alpha_1/(\alpha_1-1)}
n^{-\alpha_1/(\alpha_1-1)} .
\]
\end{proof}

\begin{lem}\label{lem:product-sum}
For all \(n\ge n_2(\omega)\),
\[
\sum_{i+j=n}
m(\delta_i^-(\omega))\,m(\delta_j^+(\sigma^i\omega))
\lesssim
(\log n)^{\frac{q\alpha_1}{\alpha_1-1}}
n^{-\frac{\alpha_1}{\alpha_1-1}} .
\]
\end{lem}

\begin{lem}\label{lem:pikovsky-P6}
The Pikovsky tower satisfies property \((P6)\). More precisely, there exists
\(M<\infty\) such that
\[
\bar m_\omega(\Delta_\omega)\le M
\]
for every \(\omega\in\Omega\).
\end{lem}

\begin{proof}
Set
\(
r_P:=\alpha_1/{(\alpha_2-1)}.
\)
By the standing assumption \(\alpha_2<\alpha_1+1\), we have \(r_P>1\).

Using the deterministic endpoint comparison with the slowest endpoint sequence
\(x_n^\pm(\alpha_2)\), together with the difference estimate for
\(\delta_n^\pm(\omega)\), we obtain
\[
m(\delta_n^\pm(\omega))
\lesssim
n^{-r_P-1},
\]
uniformly in \(n\) and \(\omega\).

By the cylinder decomposition in Lemma~\ref{lem:cyl},
\[
m(\delta_{ij}^-(\omega))
\lesssim
m(\delta_i^-(\omega))m(\delta_j^+(\sigma^i\omega))
\lesssim
i^{-r_P-1}j^{-r_P-1}.
\]
Hence
\[
\bar m_\omega(\Delta_\omega)
=
\sum_{i,j\ge1}(i+j)m(\delta_{ij}^-(\omega))
\lesssim
\sum_{i,j\ge1}(i+j)i^{-r_P-1}j^{-r_P-1}.
\]
The last series is finite because \(r_P>1\). Thus \((P6)\) holds.
\end{proof}

\begin{proof}[Proof of Proposition~\ref{Tail}]
By Lemma~\ref{lem:cyl},
\[
m(R_\omega>n)
\le
\sum_{i+j>n}
m(\delta_i^-(\omega))\,m(\delta_j^+(\sigma^i\omega)).
\]

Applying Lemma~\ref{lem:product-sum} and summing over the tail gives
\[
m(R_\omega>n)
\lesssim
(\log n)^{\frac{q\alpha_1}{\alpha_1-1}}
n^{-\frac1{\alpha_1-1}} .
\]

Let \(n_2(\omega)\) be the minimal integer such that the endpoint estimates needed in
Lemma~\ref{lem:product-sum} hold for all relevant iterates. The stretched exponential
estimate follows from the corresponding estimates for the endpoint cutoffs:
\[
\mathbb P\{n_2(\omega)>n\}
\le
\sum_{k\ge n}\mathbb P\{n_1(\sigma^k\omega)>k\}
\le
Ce^{-un^v}.
\]

Finally, since \(R_\omega=i+j\) on \(\delta_{ij}^-(\omega)\), Lemma~\ref{lem:cyl} and
Lemma~\ref{lem:product-sum} give
\[
\begin{aligned}
\int_\Omega m\{x\in\Lambda_\omega:R_\omega(x)=n\}\,d\mathbb P
&=
\int_\Omega \sum_{i+j=n}m(\delta_{ij}^-(\omega))\,d\mathbb P(\omega) \\
&\lesssim
(\log n)^{\frac{q\alpha_1}{\alpha_1-1}}
n^{-\frac{\alpha_1}{\alpha_1-1}} .
\end{aligned}
\]
This proves \((P4)\) and \((P7)\).
\end{proof}

\subsubsection{Decay of correlations}

We establish decay of correlations for the random system.

In the Pikovsky case, Lebesgue measure is fibrewise invariant. The lifted tower
measure projects to \(m\), and the random constants appearing in the tower
regularity estimates are absorbed into \(C_\omega\).

Let \(0<\gamma<1\) be the constant from Proposition \ref{distortion}.

\subsubsection*{Function spaces}

We consider the spaces
\begin{align*}
\mathcal{F}^1_\gamma
&=
\Bigl\{
\varphi_\omega:\Delta_\omega\to\mathbb R :
|\varphi_\omega(x)-\varphi_\omega(y)|
\le C_\varphi \gamma^{s(x,y)}
\Bigr\},\\
\mathcal{F}^+_\gamma
&=
\Bigl\{
\varphi_\omega:\Delta_\omega\to\mathbb R^+ :
\varphi_\omega|J_\omega>0 \ \text{or}\ \equiv 0,\;
\left|\log\frac{\varphi_\omega(x)}{\varphi_\omega(y)}\right|
\le C_\varphi \gamma^{s(x,y)}
\ \forall x,y\in J_\omega\in\mathcal P_\omega
\Bigr\}.
\end{align*}
Here \(s(x,y)\) is the separation time associated to \(\mathcal P_\omega\).

\subsubsection*{Lifted observables}

Let
\[
\varphi=(\varphi_\omega)_{\omega\in\Omega},
\qquad
\psi=(\psi_\omega)_{\omega\in\Omega}
\]
be fibre observables satisfying the assumptions of Theorem~\ref{thm2}. Define
\[
\bar\varphi_\omega=\varphi_\omega\circ\pi_\omega,
\qquad
\bar\psi_\omega=\psi_\omega\circ\pi_\omega .
\]
Then, for all \(n\ge0\),
\[
\int_{\mathbb T}
(\varphi_{\sigma^n\omega}\circ g_\omega^n)\psi_\omega\,dm
=
\int_{\Delta_\omega}
(\bar\varphi_{\sigma^n\omega}\circ F_\omega^n)
\bar\psi_\omega\,d\bar m_\omega .
\]

\subsubsection*{Regularity}

If \(x,y\in\Delta_\omega\) have separation time \(s(x,y)\), then the
expansion of the induced map gives
\[
|\pi_\omega(x)-\pi_\omega(y)|
\le C\rho^{s(x,y)}
\]
for some \(0<\rho<1\). Since
\(
\sup_{\omega}\|\psi_\omega\|_{\mathcal C^\eta}<\infty,
\)
the lifted family \(\bar\psi_\omega=\psi_\omega\circ\pi_\omega\) belongs
uniformly to the tower Hölder space.


\subsubsection*{Application of the tower formalism}

The conditions (P1)--(P7) hold for the tower. In particular:
distortion is controlled by Proposition \ref{distortion},
return time tails are given by Proposition \ref{Tail},
and the tower reference measures satisfy the required invariance properties.

Applying Theorem 4.2 of \cite{BahBosRuz19} yields quenched decay of correlations:
\[
\mathrm{Cor}_{n,\omega}(\bar\varphi,\bar\psi)
=
\int_{\Delta_\omega}
(\bar\varphi_{\sigma^n\omega}\circ F_\omega^n)
\bar\psi_\omega\,d\bar m_\omega
-
\int_{\Delta_{\sigma^n\omega}}\bar\varphi_{\sigma^n\omega}\,d\bar m_{\sigma^n\omega}
\int_{\Delta_\omega}\bar\psi_\omega\,d\bar m_\omega,
\]
with polynomial rate determined by the return time tail.

\subsubsection*{Future and past correlations}

The same argument applied in both time directions (Theorem 4.3 of
\cite{BahBosRuz19}) gives
\[
|\operatorname{Cor}_{n,\omega}(\phi,\psi)|
\le
C_\omega C_{\phi,\psi}
n^{-\left(\frac1{\alpha_1-1}-1-\delta\right)}.
\]

This completes the argument.

\section{Random Grossmann--Horner Maps}\label{secthm3}

We verify the tower properties for the random Grossmann--Horner family. The
construction is analogous to the Pikovsky case, with the return structure
adapted to the alternating dynamics between the two branches of \(h_\omega\).

\subsection{Inducing construction}\label{sec:GH-tower}

Set
\(
\Lambda_\omega=\Delta_0^-(\omega).
\)
For \(n\geq1\), define
\[
\Delta_n^-(\omega)
=
h_\omega^{-1}
\bigl(\Delta_{n-1}^-(\sigma\omega)\bigr)
\cap I^-,
\quand
\Delta_n^+(\omega)
=
h_\omega^{-1}
\bigl(\Delta_{n-1}^-(\sigma\omega)\bigr)
\cap I^+ .
\]
Then
\(\{\Delta_n^-(\omega)\}_{n\geq0}\) and
\(\{\Delta_n^+(\omega)\}_{n\geq0}\) are mod \(0\) partitions of
\(I^-\) and \(I^+\), respectively.

For \(n\geq1\), define
\[
\delta_n^-(\omega)
=
h_\omega^{-1}
\bigl(\Delta_{n-1}^+(\sigma\omega)\bigr)
\cap
\Lambda_\omega .
\]
For \(n\geq1\), these elements form a mod \(0\) partition of
\(\Lambda_\omega\), and for  \(n\geq1\),
\(
h_\omega^n:
\delta_n^-(\omega)
\longrightarrow
\Lambda_{\sigma^n\omega}
\)
is a \(C^2\) diffeomorphism.

The remaining first return element is decomposed as
\[
\delta_{1,n}^-(\omega)
=
h_\omega^{-1}
\bigl(\delta_n^+(\sigma\omega)\bigr)\cap \delta_{1}^-(\omega),
\quad n\geq1,
\]
where
\(
\delta_n^+(\omega)
=
h_\omega^{-1}
\bigl(\Delta_{n-1}^-(\sigma\omega)\bigr)
\cap \Delta_{0}^+(\omega) .
\)
Then
\[
h_\omega^{n+1}:
\delta_{1,n}^-(\omega)
\longrightarrow
\Lambda_{\sigma^{n+1}\omega}
\]
is a \(C^2\) diffeomorphism, and \(\delta_1^-(\omega)=\cup_{i\ge 1}\delta_{1,i}^-(\omega) \mod 0.\)

Define
\[
\mathcal P_\omega
=
\{\delta_n^-(\omega):n\geq2\}
\cup
\{\delta_{1,n}^-(\omega):n\geq1\}.
\]
The return time is
\[
R_\omega|_{\delta_n^-(\omega)}=n,
\qquad n\geq2,
\]
and
\(
R_\omega|_{\delta_{1,n}^-(\omega)}=n+1.
\)

\begin{prop}\label{prop:gh-basic}
The random Grossmann--Horner tower satisfies properties (P1), (P3), and
(P5).
\end{prop}

\begin{proof}
Property (P1) follows from the construction, since every element of
\(\mathcal P_\omega\) is mapped by \(h_\omega^{R_\omega}\) bijectively onto
the corresponding base \(\Lambda_{\sigma^{R_\omega}\omega}\).

The generating property (P3) follows from the uniform expansion of the
induced branches.

Finally, the return times \(2\) and \(3\) occur through
\(\delta_2^-(\omega)\) and \(\delta_{1,2}^-(\omega)\), respectively.
Their measures are uniformly bounded below by continuity of the endpoints
and compactness of the parameter range. Since
\(\gcd(2,3)=1\), property (P5) holds.
\end{proof}

\subsection{Endpoint estimates}

We verify (P2), (P4), (P6), and (P7) using estimates on the endpoint
sequences and the induced partition elements.

Define the endpoint sequences by
\[
x_n^-(\omega)=h_\omega^{-n}(0),
\qquad
x_n^+(\omega)
=
(h_\omega|_{I^+})^{-1}
(x_{n-1}^-(\sigma\omega)),
\quad n\geq1 .
\]
Then
\[
\Delta_n^-(\omega)
=
(x_{n+1}^-(\omega),x_n^-(\omega)),
\qquad
\Delta_n^+(\omega)
=
(x_n^+(\omega),x_{n+1}^+(\omega)).
\]

For the first-return partition, set
\[
y_n^-(\omega)
=
(h_\omega|_{I^-})^{-1}
(x_{n-1}^+(\sigma\omega)),
\qquad
y_n^+(\omega)
=
(h_\omega|_{I^+})^{-1}
(x_{n-1}^-(\sigma\omega)),
\]
so that
\(
\delta_n^-(\omega)
=
(y_n^-(\omega),y_{n+1}^-(\omega)).
\)

For constant parameters, the endpoint estimates from
\cite{gnsp} give

\begin{lem}\label{lem:gh-deterministic}
For the deterministic Grossmann--Horner map,
\[
1-x_n^\pm
\asymp
n^{-\frac1{\gamma-1}},
\qquad
m(\Delta_n^\pm)
\asymp
n^{-\frac{\gamma}{\gamma-1}},
\]
and
\[
|y_n^\pm|
\asymp
n^{-\frac1{k(\gamma-1)}},
\qquad
m(\delta_n^\pm)
\asymp
n^{-\frac{k(\gamma-1)+1}{k(\gamma-1)}} .
\]
\end{lem}

The random endpoints are compared with the deterministic ones. If
\(\gamma_1\leq\gamma(\omega)\leq\gamma_2\), then
\[
1-x_n^\pm(\gamma_1)
\lesssim
1-x_n^\pm(\omega)
\lesssim
1-x_n^\pm(\gamma_2).
\]

Moreover, by the local form of the inverse branches near the singular
point,
\[
|y_n^\pm(\omega)|
\asymp
(1\mp x_{n-1}^{\mp}(\sigma\omega))^{1/k(\omega)} .
\]
Consequently,
\(
m(\delta_n^\pm(\omega))
\asymp
(1\mp x_{n-1}^{\mp}(\sigma\omega))^{1/k(\omega)}
-
(1\mp x_n^{\mp}(\sigma\omega))^{1/k(\omega)} .
\)

\begin{lem}\label{lem:low-xn-gh}
Assume \ref{itm:B1}. Then there exist constants \(C,u,v>0\) and a
random variable \(n_1:\Omega\to\mathbb N\), with
\(
\mathbb P\{n_1>n\}\le Ce^{-un^v},
\)
such that for all \(n\ge n_1(\omega)\),
\[
(1\mp x_n^\pm(\omega))^{\gamma_1-1}
\le
C\frac{(\log n)^q}{n}.
\]
\end{lem}

\begin{proof}
We prove the estimate for the negative endpoint; the positive endpoint follows
similarly. By Assumption~\((B)\), there exist \(c_B>0\) and a random variable
\(N_B:\Omega\to\mathbb N\) satisfying
\[
\mathbb P\{N_B>n\}\le C_Be^{-u_Bn^{v_B}},
\]
such that, for all \(n\ge N_B(\omega)\),
\[
\frac{(\log n)^q}{n}
\sum_{k=1}^n B_{n,k}(\omega)
\ge c_B .
\]
On the other hand, the endpoint recursion gives
\(
\frac1{(1+x_n^-(\omega))^{\gamma_1-1}}
\ge
\sum_{k=1}^n B_{n,k}(\omega).
\)
Therefore, for \(n\ge N_B(\omega)\),
\(
\frac{(\log n)^q}{n}
\frac1{(1+x_n^-(\omega))^{\gamma_1-1}}
\ge c_B ,
\)
and hence
\[
(1+x_n^-(\omega))^{\gamma_1-1}
\le
\frac1{c_B}\frac{(\log n)^q}{n}.
\]
The estimate for the positive endpoint follows similarly. The tail estimate for
\(n_1\) follows directly from the stretched-exponential tail of \(N_B\).
\end{proof}

\begin{cor}\label{cor:y_n_scaling}
Under the assumptions of Lemma~\ref{lem:low-xn-gh},
\[
|y_n^\pm(\omega)|
\le
C\left(\frac{(\log n)^q}{n}\right)^{1/\zeta},
\qquad
\zeta=k_2(\gamma_1-1).
\]
for all \(n\geq n_1(\omega)\).
\end{cor}
\begin{proof}
By the singular inverse estimate,
\[
|y_n^\pm(\omega)|
\asymp
(1\mp x_{n-1}^{\mp}(\sigma\omega))^{1/k(\omega)}.
\]
Since \(k(\omega)\le k_2\), Lemma~\ref{lem:low-xn-gh} gives
\[
|y_n^\pm(\omega)|
\le
C\left(\frac{(\log n)^q}{n}\right)^{1/(k_2(\gamma_1-1))}
=
C\left(\frac{(\log n)^q}{n}\right)^{1/\zeta}.
\]
Replacing the cutoff by its maximum with its shifted version preserves the
stretched exponential tail.
\end{proof}

\subsection{Distortion Estimates}\label{sec:dist'}

Condition (P2) is verified first for the induced map

Define the induced map pointwise by
\[
\widehat h_\omega(x)
=
h_\omega^{R_\omega(x)}(x),
\qquad x\in\Lambda_\omega.
\]
For every \(J_\omega\in\mathcal P_\omega\), the return time is constant
on \(J_\omega\), and
\[
\widehat h_\omega|_{J_\omega}
=
h_\omega^{R_\omega(J_\omega)}
:
J_\omega
\longrightarrow
\Lambda_{\sigma^{R_\omega(J_\omega)}\omega}.
\]

Because the tower map has unit Jacobian on every intermediate level, the
Jacobian of \(F_\omega^{R_\omega}\) on the base coincides with the Jacobian of
\(\widehat h_\omega\).

The two types of return branches have the following itineraries.

\[\underline{\text{For } J_\omega=\delta_{1,n}^-(\omega) }\text{ : } 
\delta_{1,n}^-(\omega)
\xrightarrow{\ h_\omega\ }
\delta_n^+(\sigma\omega)
\xrightarrow{\ h_{\sigma\omega}\ }
\Delta_{n-1}^-(\sigma^2\omega)
\xrightarrow{\ h_{\sigma^2\omega}\ }
\Delta_{n-2}^-(\sigma^3\omega)
\longrightarrow\cdots\longrightarrow
\Lambda_{\sigma^{n+1}\omega}.
\]
\[\underline{\text{For } J_\omega=\delta_n^-(\omega) }\text{ : } 
\delta_n^-(\omega)
\xrightarrow{\ h_\omega\ }
\Delta_{n-1}^+(\sigma\omega)
\xrightarrow{\ h_{\sigma\omega}\ }
\Delta_{n-2}^-(\sigma^2\omega)
\xrightarrow{\ h_{\sigma^2\omega}\ }
\Delta_{n-3}^-(\sigma^3\omega)
\longrightarrow\cdots\longrightarrow
\Lambda_{\sigma^n\omega}.
\]
The distortion estimate follows by decomposing each return branch according
to the above itineraries. The initial transition step, when present, lies in
a fixed compact region and has uniformly bounded nonlinearity. The singular
step is estimated in coordinates measuring distance from the singular point
and the corresponding neutral endpoint. The subsequent neutral excursion is
treated in endpoint distance coordinates, where the one step distortion
bounds telescope along the excursion. This yields a Euclidean distortion
estimate in terms of the length of the final image interval. Uniform expansion
of the induced branches then converts this estimate into the separation time
bound required in \((P2)\).

\begin{lem}\label{ustb}
There exists \(\varepsilon>0\) such that, for every \(\omega\in\Omega\),
\(
\operatorname{dist}
\left(
\delta_1^-(\omega),
\{0,-1\}
\right)
\geq \varepsilon .
\)
Consequently,
\[
\sup_{\omega\in\Omega}
\sup_{x\in h_\omega^{-1}(\Delta_0^+(\sigma\omega))\cap\Lambda_\omega}
\left|
\frac{D^2h_\omega(x)}{Dh_\omega(x)}
\right|
<\infty .
\]
\end{lem}

\begin{proof}

By construction,
\(
h_\omega(\delta_1^-(\omega))=\Delta_0^+(\sigma\omega).
\)

We first show that \(\delta_1^-(\omega)\) is uniformly separated from \(0\).
The intervals \(\Delta_0^+(\omega)\) are uniformly separated from the
neutral endpoint \(1\). Indeed, their endpoints depend continuously on
the parameter and the parameter range is compact. Hence there exists
\(\eta>0\) such that
\[
\Delta_0^+(\omega)\subset (0,1-\eta]
\qquad
\text{for every }\omega\in\Omega .
\]
On the other hand, the local form of \(h_\omega\) on the negative side
of the singular point gives, uniformly in \(\omega\),
\[
h_\omega(x)
=
1-b(\omega)|x|^{k(\omega)}+o(|x|^{k(\omega)})
\qquad
\text{as }x\to0^- .
\]
Since \(b(\omega)\) is bounded away from \(0\) and infinity, and
\(k(\omega)\in[k_1,k_2]\subset(0,1)\), there exists
\(\varepsilon_0>0\) such that
\[
h_\omega(x)>1-\eta
\qquad
\text{for every }\omega\in\Omega
\text{ and every }x\in(-\varepsilon_0,0).
\]
Thus
\(
\delta_1^-(\omega)\cap(-\varepsilon_0,0)=\varnothing
\)
for every \(\omega\).

We next show that \(\delta_1^-(\omega)\) is uniformly separated from \(-1\).
Since \(\Delta_0^+(\omega)\subset I^+\), it is enough to use the fact
that points sufficiently close to \(-1\) remain in \(I^-\) after one
iterate. More precisely, the local form near \(-1\) gives
\[
h_\omega(x)
=
x+a(\omega)|x+1|^{\gamma(\omega)}
+o(|x+1|^{\gamma(\omega)})
\quad
\text{as }x\to -1^+ ,
\]
uniformly in \(\omega\). Therefore there exists
\(\varepsilon_1>0\) such that
\[
h_\omega(x)<0
\quad
\text{for every }\omega\in\Omega
\text{ and every }x\in(-1,-1+\varepsilon_1).
\]
Since \(h_\omega(\delta_1^-(\omega))=\Delta_0^+(\sigma\omega)\subset I^+\), this
implies
\(
\delta_1^-(\omega)\cap(-1,-1+\varepsilon_1)=\varnothing
\)
for every \(\omega\).

Let
\(
\varepsilon=\min\{\varepsilon_0,\varepsilon_1\}.
\)
Then
\[
\delta_1^-(\omega)\subset[-1+\varepsilon,-\varepsilon]
\]
for every \(\omega\), and hence
\(
\operatorname{dist}(\delta_1^-(\omega),\{0,-1\})\geq \varepsilon .
\)

It remains to bound the nonlinearity. The family \(h_\omega\) is
uniformly \(C^2\) on compact subsets of \(I\setminus\{-1,0,1\}\), and
\[
[-1+\varepsilon,-\varepsilon]
\subset I\setminus\{-1,0,1\}.
\]
Moreover, by the expansion assumption, \(Dh_\omega\) is bounded away
from zero on this compact set, uniformly in \(\omega\). Therefore,
by compactness of the parameter range,
\[
\sup_{\omega\in\Omega}
\sup_{x\in[-1+\varepsilon,-\varepsilon]}
\left|
\frac{D^2h_\omega(x)}{Dh_\omega(x)}
\right|
<\infty .
\]
Since \(\delta_1^-(\omega)\subset[-1+\varepsilon,-\varepsilon]\), the claimed
bound follows.
\end{proof}

We now give the singular distortion contribution near the origin. 
\begin{lem}
\label{lem:GH-singular-step}
There exist constants \(\varepsilon_1>0\) and \(C>0\) such that the
following holds. For \(\omega\in\Omega\) and either choice of sign,
define
\[
\Psi_\omega^\pm(t)
:=
1-h_\omega(\pm t),
\qquad
0<t\leq\varepsilon_1.
\]
Then, for every \(0<u<v\leq\varepsilon_1\),
\[
\left|
\log
\frac{D\Psi_\omega^\pm(v)}
     {D\Psi_\omega^\pm(u)}
\right|
\leq
C\log
\frac{\Psi_\omega^\pm(v)}
     {\Psi_\omega^\pm(u)}.
\]
The constants \(\varepsilon_1\) and \(C\) are independent of
\(\omega\) and of the choice of sign.
\end{lem}

\begin{proof}
Fix \(\omega\in\Omega\) and a choice of sign, and write
\(
S=\Psi_\omega^\pm,
\qquad
k=k(\omega),
\qquad
b=b(\omega).
\)
The uniform local expansions at the singular point give
\[
\Psi(t)=bt^k+o(t^k),
\quad
D\Psi(t)=bkt^{k-1}+o(t^{k-1}),
\quand 
D^2\Psi(t)=bk(k-1)t^{k-2}+o(t^{k-2}),
\]
uniformly in \(\omega\) and in the choice of sign. Since
\[
b(\omega)\in[b_1,b_2],
\qquad
k(\omega)\in[k_1,k_2]\subset(0,1),
\]
we may choose \(\varepsilon_1>0\), independently of \(\omega\), such
that
\[
\frac{D\Psi(t)}{\Psi(t)}
\geq
\frac{c_0}{t},
\qquad
\left|
\frac{D^2\Psi(t)}{D\Psi(t)}
\right|
\leq
\frac{C_0}{t}
\]
for every \(0<t\leq\varepsilon_1\), where \(c_0,C_0>0\) are uniform.

It follows that
\[
\left|
\frac{D^2\Psi(t)}{D\Psi(t)}
\right|
\leq
C\frac{D\Psi(t)}{\Psi(t)}
\]
for some uniform \(C>0\). Therefore,
\[
\left|
\log
\frac{D\Psi(v)}{D\Psi(u)}
\right|
\leq
\int_u^v
\left|
\frac{D^2\Psi(t)}{D\Psi(t)}
\right|\,dt
\leq
C\int_u^v
\frac{D\Psi(t)}{\Psi(t)}\,dt
=
C\log\frac{\Psi(v)}{\Psi(u)}.
\]
\end{proof}

We now give the distortion contribution near the neutral points.
\begin{lem}
\label{lem:GH-neutral-step}
There exist constants \(\varepsilon_0>0\) and \(C>0\) such that the
following holds. For \(\omega\in\Omega\) and either choice of sign, define
\[
\Phi_\omega^\pm(t)
:=
1+h_\omega\bigl(\pm(1-t)\bigr),
\qquad
0<t\leq\varepsilon_0.
\]
Then, for every \(0<u<v\leq\varepsilon_0\),
\[
\left|
\log
\frac{D\Phi_\omega^\pm(v)}
     {D\Phi_\omega^\pm(u)}
\right|
\leq
C\left(
\log
\frac{\Phi_\omega^\pm(v)}
     {\Phi_\omega^\pm(u)}
-
\log\frac{v}{u}
\right).
\]
The constants \(\varepsilon_0\) and \(C\) are independent of
\(\omega\) and of the choice of sign.
\end{lem}

\begin{proof}
Fix \(\omega\in\Omega\) and a choice of sign. To simplify notation,
write
\[
\Phi=\Phi_\omega^\pm,
\qquad
\gamma=\gamma(\omega),
\qquad
a=a(\omega).
\]
The uniform local expansions of \(h_\omega\) at the neutral endpoints
give
\[
\Phi(t)=t+at^\gamma+o(t^\gamma),
\]
uniformly in \(\omega\) and in the choice of sign. Hence, after
decreasing \(\varepsilon_0>0\) if necessary, there exist uniform
constants \(c_1,c_2>0\) such that
\[
c_1t^\gamma
\leq
\Phi(t)-t
\leq
c_2t^\gamma,
\qquad
0<t\leq\varepsilon_0.
\]
In particular, we may assume that
\(
t\leq\Phi(t)\leq2t.
\)

Set
\(
\phi(t):=\Phi(t)-t.
\)
Since \(h_\omega\) is convex on each monotonicity interval, \(\Phi\),
and hence \(\phi\), is convex. Choose \(0<\vartheta<1\) sufficiently
small that
\(
c_1\vartheta-c_2\vartheta^{\gamma_1}>0.
\)
By convexity,
\[
\phi'(t)
\geq
\frac{\phi(t)-\phi(\vartheta t)}
     {(1-\vartheta)t}.
\]
Therefore,
\[
\begin{aligned}
tD\Phi(t)-\Phi(t)
=
t\phi'(t)-\phi(t)\geq
\frac{\vartheta\phi(t)-\phi(\vartheta t)}
     {1-\vartheta}\geq
\frac{c_1\vartheta-c_2\vartheta^\gamma}
     {1-\vartheta}t^\gamma\geq
c_0t^\gamma
\end{aligned}
\]
for some uniform constant \(c_0>0\), since
\(\gamma\geq\gamma_1\).

The uniform second derivative estimates near the neutral endpoints
give
\[
|D^2\Phi(t)|
\leq
C_0t^{\gamma-2}.
\]
Moreover, the expansion assumption implies
\(
D\Phi(t)\geq1.
\)
It follows that
\[
\frac{d}{dt}\log\frac{\Phi(t)}{t}
=
\frac{tD\Phi(t)-\Phi(t)}{t\Phi(t)}
\geq
\frac{c_0}{2}t^{\gamma-2},
\]
whereas
\[
\left|
\frac{D^2\Phi(t)}{D\Phi(t)}
\right|
\leq
C_0t^{\gamma-2}.
\]
Consequently,
\[
\left|
\frac{D^2\Phi(t)}{D\Phi(t)}
\right|
\leq
C\frac{d}{dt}\log\frac{\Phi(t)}{t}.
\]
Integrating between \(u\) and \(v\) gives
\[
\begin{aligned}
\left|
\log\frac{D\Phi(v)}{D\Phi(u)}
\right|
\leq
\int_u^v
\left|
\frac{D^2\Phi(t)}{D\Phi(t)}
\right|\,dt\leq
C\int_u^v
\frac{d}{dt}\log\frac{\Phi(t)}{t}\,dt=
C\left(
\log\frac{\Phi(v)}{\Phi(u)}
-
\log\frac{v}{u}
\right).
\end{aligned}
\]
\end{proof}

We now give the distortion contribution for excursion near the neutral region.
\begin{cor}
\label{cor:GH-neutral-excursion}
There exist constants \(d_0>0\) and \(C>0\) such that the following
holds. Let \(J_\omega\in\mathcal P_\omega\), let \(x,y\in J_\omega\),
and set
\[
E_i=h_\omega^i([x,y]).
\]
Suppose that \(E_p,\ldots,E_q\) form a consecutive neutral excursion.

For each \(p\leq i\leq q\), let
\(\xi_i\in\{-1,1\}\) be the neutral endpoint whose fixed
neighbourhood contains \(E_i\). Let \(\xi_{q+1}\) be the endpoint
associated with the image of the last neutral branch, and define
\[
d_i(z):=|z-\xi_i|,
\quad
p\leq i\leq q+1.
\]
\[\text{Write }\quad
d_i(E_i)=[u_i,v_i],
\quad
0<u_i<v_i,
\quad\text{ and set }\quad
Q_i:=\log\frac{v_i}{u_i}.
\]

Then
\[
\sum_{i=p}^{q}
\left|
\log
\frac{
\left|Dh_{\sigma^i\omega}
       \bigl(h_\omega^i(x)\bigr)\right|
}{
\left|Dh_{\sigma^i\omega}
       \bigl(h_\omega^i(y)\bigr)\right|
}
\right|
\leq
C\bigl(Q_{q+1}-Q_p\bigr).
\]

If, in addition, the exit interval
\(
E_{\mathrm{out}}:=E_{q+1}
\)
satisfies
\(
\operatorname{dist}
\bigl(E_{\mathrm{out}},\{-1,1\}\bigr)
\geq d_0,
\)
then
\[
\sum_{i=p}^{q}
\left|
\log
\frac{
\left|Dh_{\sigma^i\omega}
       \bigl(h_\omega^i(x)\bigr)\right|
}{
\left|Dh_{\sigma^i\omega}
       \bigl(h_\omega^i(y)\bigr)\right|
}
\right|
\leq
C\,m(E_{\mathrm{out}}).
\]
\end{cor}

\begin{proof}
For each \(p\leq i\leq q\), let
\[
\Phi_i
:=
d_{i+1}\circ h_{\sigma^i\omega}\circ d_i^{-1}
\]
be the expression of \(h_{\sigma^i\omega}\) in the distance
coordinates associated with \(E_i\) and \(E_{i+1}\). According to
whether \(E_i\) lies near \(1\) or near \(-1\), this is one of the
maps \(\Phi_{\sigma^i\omega}^{\pm}\) from
Lemma~\ref{lem:GH-neutral-step}.

Since
\[
E_{i+1}=h_{\sigma^i\omega}(E_i)
\qquad
d_i(E_i)=[u_i,v_i],
\quand
d_{i+1}(E_{i+1})=[u_{i+1},v_{i+1}],
\]
we have
\(
[u_{i+1},v_{i+1}]
=
\Phi_i([u_i,v_i]).
\)
The map \(\Phi_i\) is increasing in distance coordinates, and hence
\[
u_{i+1}=\Phi_i(u_i),
\qquad
v_{i+1}=\Phi_i(v_i).
\]

Set
\(
x_i=h_\omega^i(x),
\qquad
y_i=h_\omega^i(y).
\)
Since \(d_i\) and \(d_{i+1}\) are affine isometries, their derivatives
have modulus one. Therefore
\[
D\Phi_i\bigl(d_i(z)\bigr)
=
\left|Dh_{\sigma^i\omega}(z)\right|
\]
for \(z\in E_i\). Moreover,
\[
\{d_i(x_i),d_i(y_i)\}=\{u_i,v_i\}.
\]
Consequently,
\[
\left|
\log
\frac{
\left|Dh_{\sigma^i\omega}(x_i)\right|
}{
\left|Dh_{\sigma^i\omega}(y_i)\right|
}
\right|
=
\left|
\log
\frac{D\Phi_i(v_i)}
     {D\Phi_i(u_i)}
\right|.
\]

By Lemma~\ref{lem:GH-neutral-step},
\[
\begin{aligned}
\left|
\log
\frac{D\Phi_i(v_i)}
     {D\Phi_i(u_i)}
\right|
\leq
C\left(
\log\frac{\Phi_i(v_i)}{\Phi_i(u_i)}
-
\log\frac{v_i}{u_i}
\right)
=
C\left(
\log\frac{v_{i+1}}{u_{i+1}}
-
\log\frac{v_i}{u_i}
\right)=
C(Q_{i+1}-Q_i).
\end{aligned}
\]
Summing over \(i=p,\ldots,q\) gives
\[
\begin{aligned}
\sum_{i=p}^{q}
\left|
\log
\frac{
\left|Dh_{\sigma^i\omega}
       \bigl(h_\omega^i(x)\bigr)\right|
}{
\left|Dh_{\sigma^i\omega}
       \bigl(h_\omega^i(y)\bigr)\right|
}
\right|
\leq
C\sum_{i=p}^{q}(Q_{i+1}-Q_i)=
C(Q_{q+1}-Q_p).
\end{aligned}
\]

Suppose now that
\[
\operatorname{dist}
\bigl(E_{\mathrm{out}},\{-1,1\}\bigr)
\geq d_0,
\qquad
E_{\mathrm{out}}=E_{q+1}.
\]
Since \(d_{q+1}\) is distance from one of the neutral endpoints, this
implies
\[
u_{q+1}\geq d_0.
\]
Therefore,
\[
\begin{aligned}
Q_{q+1}
=
\log\frac{v_{q+1}}{u_{q+1}}=
\log\left(
1+\frac{v_{q+1}-u_{q+1}}{u_{q+1}}
\right)\leq
\frac{v_{q+1}-u_{q+1}}{u_{q+1}}=
\frac{m(E_{\mathrm{out}})}{u_{q+1}}\leq
d_0^{-1}m(E_{\mathrm{out}}).
\end{aligned}
\]
Since \(Q_p\geq0\), it follows that
\[
Q_{q+1}-Q_p
\leq
Q_{q+1}
\leq
d_0^{-1}m(E_{\mathrm{out}}),
\]
which proves the second assertion.
\end{proof}

We first establish a distortion estimate for the induced branches in terms of the Euclidean distance between their images.
\begin{prop}
\label{prop:GH-euclidean-distortion}
There exists \(C>0\) such that, for every \(\omega\in\Omega\), every
\(J_\omega\in\mathcal P_\omega\), and every \(x,y\in J_\omega\),
\[
\left|
\log
\frac{
\left|Dh_\omega^{R_\omega(J_\omega)}(x)\right|
}{
\left|Dh_\omega^{R_\omega(J_\omega)}(y)\right|
}
\right|
\leq
C\left|
h_\omega^{R_\omega(J_\omega)}(x)
-
h_\omega^{R_\omega(J_\omega)}(y)
\right|.
\]
\end{prop}

\begin{proof}
Fix \(J_\omega\in\mathcal P_\omega\), and set
\(
R:=R_\omega(J_\omega),
\qquad
x_i:=h_\omega^i(x),
\qquad
y_i:=h_\omega^i(y).
\)
Let \(E_i\) be the interval with endpoints \(x_i\) and \(y_i\). By
the chain rule,
\begin{equation}\label{dist_sum5.9}
    \left|
\log
\frac{|Dh_\omega^R(x)|}{|Dh_\omega^R(y)|}
\right|
\leq
\sum_{i=0}^{R-1}
\left|
\log
\frac{
|Dh_{\sigma^i\omega}(x_i)|
}{
|Dh_{\sigma^i\omega}(y_i)|
}
\right|.
\end{equation}

Let
\(
U_{\mathrm{neu}}
:=
U_{1^-}\cup U_{-1^+}.
\) We first consider
\(
J_\omega=\delta_n^-(\omega),
R=n.
\)
In this case, \(i=0\) is the singular step and the neutral excursion
begins at
\(
p=1.
\)

Let \(q\geq1\) be maximal such that
\(
E_i\subset U_{\mathrm{neu}},
1\leq i\leq q.
\)
Thus
\(
E_{\mathrm{out}}:=E_{q+1}
\)
is the first interval after the maximal neutral block.

Let
\[
d_1(E_1)=[u_1,v_1],
\qquad
Q_1:=\log\frac{v_1}{u_1},
\]
where \(d_1\) is distance from the neutral endpoint corresponding to
\(E_1\). Lemma~\ref{lem:GH-singular-step} gives
\[
\left|
\log
\frac{|Dh_\omega(x)|}{|Dh_\omega(y)|}
\right|
\leq
CQ_1.
\]
Corollary~\ref{cor:GH-neutral-excursion}, applied with \(p=1\), gives
\[
\sum_{i=1}^{q}
\left|
\log
\frac{
|Dh_{\sigma^i\omega}(x_i)|
}{
|Dh_{\sigma^i\omega}(y_i)|
}
\right|
\leq
C(Q_{q+1}-Q_1).
\]
Consequently, the singular and neutral contributions are bounded by
\(
CQ_1+C(Q_{q+1}-Q_1)
=
CQ_{q+1}.
\)

By the uniform geometry of the return branches,
\(E_{q+1}\) is uniformly separated from the neutral endpoints.
Therefore,
\(
Q_{q+1}\lesssim m(E_{q+1}).
\)

Moreover, all subsequent images before the return lie in a fixed
compact subset of
\(
I\setminus\{-1,0,1\},
\)
and their number is uniformly bounded. Hence
\[
\left|
\log
\frac{
|Dh_{\sigma^i\omega}(x_i)|
}{
|Dh_{\sigma^i\omega}(y_i)|
}
\right|
\leq
C\,m(E_i)
\]
for every remaining index \(i\). Since
\[
|Dh_{\sigma^i\omega}|>1,
\]
interval lengths do not decrease along the remaining orbit, and thus
\(
m(E_i)\leq m(E_R).
\)
It follows that all the remaining contributions in \eqref{dist_sum5.9} are
bounded by \(C\,m(E_R)\). Therefore,
\[
\left|
\log
\frac{|Dh_\omega^R(x)|}{|Dh_\omega^R(y)|}
\right|
\leq
C\,m(E_R).
\]

We now consider
\(
J_\omega=\delta_{1,n}^-(\omega),
\quad
R=n+1.
\)
In this case, \(i=0\) is the transition step, \(i=1\) is the
singular step, and the neutral excursion begins at
\(
p=2.
\)

Since
\(
J_\omega
\subset
\delta_1^-(\omega)
=
h_\omega^{-1}
\bigl(\Delta_0^+(\sigma\omega)\bigr)
\cap\Lambda_\omega,
\)
Lemma~\ref{ustb} and the mean value theorem give
\[
\left|
\log
\frac{|Dh_\omega(x)|}{|Dh_\omega(y)|}
\right|
\leq
C\,m(E_0).
\]
The derivative is uniformly bounded away from one on
\(\delta_1^-(\omega)\), and all subsequent derivatives have modulus
greater than one. Hence
\[
m(E_0)\lesssim m(E_1)\leq m(E_R),
\]
so the transition contribution is bounded by \(C\,m(E_R)\).

The interval \(E_1\) lies on a singular branch, and its image \(E_2\)
lies in a neutral neighbourhood. Let \(q\geq2\) be maximal such that
\(
E_i\subset U_{\mathrm{neu}},
\quad
2\leq i\leq q.
\)
Thus
\(
E_{\mathrm{out}}:=E_{q+1}
\)
is the first interval after the maximal neutral block.

Write
\(
d_2(E_2)=[u_2,v_2],
Q_2:=\log\frac{v_2}{u_2}.
\)
Applying Lemma~\ref{lem:GH-singular-step} to
\(h_{\sigma\omega}\) on \(E_1\) gives
\[
\left|
\log
\frac{
|Dh_{\sigma\omega}(x_1)|
}{
|Dh_{\sigma\omega}(y_1)|
}
\right|
\leq
CQ_2.
\]
Corollary~\ref{cor:GH-neutral-excursion}, applied with \(p=2\), gives
\[
\sum_{i=2}^{q}
\left|
\log
\frac{
|Dh_{\sigma^i\omega}(x_i)|
}{
|Dh_{\sigma^i\omega}(y_i)|
}
\right|
\leq
C(Q_{q+1}-Q_2).
\]
Thus the singular and neutral contributions are bounded by
\(
CQ_2+C(Q_{q+1}-Q_2)
=
CQ_{q+1}.
\)
As in the first case,
\[
Q_{q+1}\lesssim m(E_{q+1})\leq m(E_R),
\]
and the remaining compact contributions are bounded by
\(C\,m(E_R)\). Combining the transition, singular, neutral, and
compact contributions yields
\[
\left|
\log
\frac{|Dh_\omega^R(x)|}{|Dh_\omega^R(y)|}
\right|
\leq
C\,m(E_R).
\]

Finally,
\(
m(E_R)
=
|x_R-y_R|
=
\left|
h_\omega^R(x)-h_\omega^R(y)
\right|,
\)
which proves the claim.

The finitely many return atoms whose itineraries do not enter the
chosen local singular or neutral neighbourhoods are handled by the
same argument using uniform \(C^2\) bounds and compactness.
\end{proof}

We now give uniform expansion estimates for the induced branches.
\begin{lem}
\label{lem:GH-induced-expansion}
There exists \(\lambda>1\) such that, for every \(\omega\in\Omega\),
every \(J_\omega\in\mathcal P_\omega\), and every \(x\in J_\omega\),
\(
\left|Dh_\omega^{R_\omega}(x)\right|
\geq
\lambda .
\)
\end{lem}

\begin{proof}
The bases \(\Lambda_\omega\) are uniformly separated from the neutral
endpoint \(-1\). Near the singular point \(0\),
\[
|Dh_\omega(x)|
\asymp
|x|^{k(\omega)-1}
\longrightarrow\infty
\]
uniformly over the parameter range. Hence there exists
\(\varepsilon>0\) such that
\[
|Dh_\omega(x)|\geq2
\qquad
\text{for }x\in\Lambda_\omega\cap(-\varepsilon,0).
\]
On the remaining compact part of the bases, continuity in the
parameter and the assumption \(|Dh_\omega|>1\) give
\[
\inf_{\omega\in\Omega}
\inf_{x\in\Lambda_\omega\setminus(-\varepsilon,0)}
|Dh_\omega(x)|
>1.
\]
Thus there exists \(\lambda>1\) such that
\(
|Dh_\omega(x)|\geq\lambda
\qquad
\text{for every }x\in\Lambda_\omega.
\)
Every subsequent derivative along a return branch has modulus greater
than one, so
\(
|Dh_\omega^{R_\omega}(x)|\geq\lambda.
\)
\end{proof}

We now give the distortion for the induced Grossmann--Horner map; condition (P2).
\begin{prop}
\label{distortion'}
There exist constants \(D>0\) and \(0<\rho<1\) such that, for every
\(\omega\in\Omega\), every \(J_\omega\in\mathcal P_\omega\), and every
\(x,y\in J_\omega\),
\[
\left|
\frac{
JF_\omega^{R_\omega}(x)
}{
JF_\omega^{R_\omega}(y)
}
-1
\right|
\leq
D\,
\rho^{
s\left(
F_\omega^{R_\omega}(x),
F_\omega^{R_\omega}(y)
\right)
}.
\]
\end{prop}

\begin{proof}
Set
\(
z=h_\omega^{R_\omega}(x),
\qquad
z'=h_\omega^{R_\omega}(y).
\)
By Lemma~\ref{lem:GH-induced-expansion}, the inverse branches of the
induced map contract uniformly. Hence, for some \(0<\rho<1\),
\(
|z-z'|
\leq
C_0
\rho^{
s\left(
F_\omega^{R_\omega}(x),
F_\omega^{R_\omega}(y)
\right)
}.
\)
Proposition~\ref{prop:GH-euclidean-distortion} therefore gives
\[
\left|
\log
\frac{
|Dh_\omega^{R_\omega}(x)|
}{
|Dh_\omega^{R_\omega}(y)|
}
\right|
\leq
C_1
\rho^{
s\left(
F_\omega^{R_\omega}(x),
F_\omega^{R_\omega}(y)
\right)
}.
\]
Since the tower map has unit Jacobian on intermediate levels,
\[
JF_\omega^{R_\omega}
=
|Dh_\omega^{R_\omega}|
\]
on the base. Finally, using
\(
|e^t-1|\leq e^{|t|}|t|,
\)
and enlarging the constant if necessary, we obtain
\[
\left|
\frac{
JF_\omega^{R_\omega}(x)
}{
JF_\omega^{R_\omega}(y)
}
-1
\right|
\leq
D\,
\rho^{
s\left(
F_\omega^{R_\omega}(x),
F_\omega^{R_\omega}(y)
\right)
}.
\]
\end{proof}

\subsection{Tail Estimate}\label{sec:tail2}

We verify conditions \((P4)\), \((P6)\), and \((P7)\) for the fibred tower
\((\Delta,F)\). The quenched and annealed tail estimates are proved first; the
uniform finiteness condition \((P6)\) is verified separately.

Recall that the return time satisfies
\[
R_\omega(x)=n
\quad\text{on }\delta_n^-(\omega),
\qquad
R_\omega(x)=n+1
\quad\text{on }\delta_{1,n}^-(\omega).
\]

\begin{rem}
The estimates on \(m(\delta_n^\pm(\omega))\) below require bounds on
differences of the singular preimage endpoints \(y_n^\pm(\omega)\), not only
bounds on the endpoint sizes \(|y_n^\pm(\omega)|\). We keep this distinction
explicit in the proof.
\end{rem}

\begin{prop}\label{Tail'}
There exist constants \(C>0\), \(u>0\), \(v\in(0,1)\), a full measure set
\(G\subset\Omega\), and a random variable \(n_1:G\to\mathbb N\) such that for all
\(\omega\in G\) and \(n>n_1(\omega)\),
\[
m(R_\omega>n)
\lesssim
\frac{(\log n)^{\frac{q}{k_2(\gamma_1-1)}}}
{n^{\frac{1}{k_2(\gamma_1-1)}}}.
\]
Moreover,
\(
\mathbb P\{n_1(\omega)>n\}\le Ce^{-un^v}.
\)
\end{prop}

\begin{proof}
We proceed in three steps.

\medskip
\noindent
\textbf{Step 1: one dimensional tail estimates.}

By Corollary~\ref{cor:y_n_scaling} and the parameter bounds
\[
\gamma(\omega)\in[\gamma_1,\gamma_2],
\qquad
k(\omega)\in[k_1,k_2],
\]
there exists a full measure set \(G\) and a random variable \(n_1(\omega)\)
such that for all \(\omega\in G\) and \(n>n_1(\omega)\),
\[
m(\delta_n^-(\omega))
\lesssim
\frac{(\log n)^{\frac{q}{k_2(\gamma_1-1)}}}
{n^{\frac{1}{k_2(\gamma_1-1)}+1}},
\qquad
m(\delta_{1,n}^-(\omega))
\lesssim
\frac{(\log n)^{\frac{q}{k_2(\gamma_1-1)}}}
{n^{\frac{1}{k_2(\gamma_1-1)}+1}}.
\]

\medskip
\noindent
\textbf{Step 2: estimate of the return time tail.}

Using the definition of the return time,
\(
m(R_\omega>n)
=
\sum_{\ell>n}m(\delta_\ell^-(\omega))
+
\sum_{\ell+1>n}m(\delta_{1,\ell}^-(\omega)).
\)
Applying the estimates from Step 1, we obtain
\[
m(R_\omega>n)
\lesssim
\sum_{\ell>n}
\frac{(\log \ell)^{\frac{q}{k_2(\gamma_1-1)}}}
{\ell^{\frac{1}{k_2(\gamma_1-1)}+1}}.
\]
Since
\(
\frac{1}{k_2(\gamma_1-1)}>1,
\)
the tail sum gives
\[
m(R_\omega>n)
\lesssim
\frac{(\log n)^{\frac{q}{k_2(\gamma_1-1)}}}
{n^{\frac{1}{k_2(\gamma_1-1)}}}.
\]

\medskip
\noindent
\textbf{Step 3: control of the random cutoff.}

Let \(n_1(\omega)\) be the minimal time after which the estimates in Step 1
hold. The stretched exponential tail of \(n_1\) follows from the corresponding
estimate in Lemma~\ref{lem:low-xn-gh}. This proves \((P4)\).
\end{proof}

Next we have the annealed return time tails
\begin{cor}\label{cor:annealed-tail}
Under the assumptions of Proposition~\ref{Tail'},
\[
\int_\Omega m\{x\in\Lambda_\omega:R_\omega(x)=n\}\,d\mathbb P
\lesssim
\frac{(\log n)^{\frac{q}{k_2(\gamma_1-1)}}}
{n^{\frac{1}{k_2(\gamma_1-1)}+1}}.
\]
\end{cor}

\begin{proof}
Since
\(
R_\omega=n
\quad\text{on }\delta_n^-(\omega),
\qquad
R_\omega=n+1
\quad\text{on }\delta_{1,n}^-(\omega),
\)
we have
\[
m\{R_\omega=n\}
=
m(\delta_n^-(\omega))
+
m(\delta_{1,n-1}^-(\omega)).
\]
On the set \(\{n_1(\omega)\le n-1\}\), the one dimensional estimates from the
proof of Proposition~\ref{Tail'} give
\[
m\{R_\omega=n\}
\lesssim
\frac{(\log n)^{\frac{q}{k_2(\gamma_1-1)}}}
{n^{\frac{1}{k_2(\gamma_1-1)}+1}}.
\]
On the complementary set, we use the trivial bound \(m\{R_\omega=n\}\le1\) and
the stretched exponential estimate for \(n_1\). Hence
\[
\begin{aligned}
\int_\Omega m\{R_\omega=n\}\,d\mathbb P
&\lesssim
\frac{(\log n)^{\frac{q}{k_2(\gamma_1-1)}}}
{n^{\frac{1}{k_2(\gamma_1-1)}+1}}
+
\mathbb P\{n_1(\omega)>n-1\} \\
&\lesssim
\frac{(\log n)^{\frac{q}{k_2(\gamma_1-1)}}}
{n^{\frac{1}{k_2(\gamma_1-1)}+1}}.
\end{aligned}
\]
This proves \((P7)\).
\end{proof}

Next we have the uniform tower finiteness.
\begin{lem}\label{lem:GH-P6}
Assume the standing finiteness condition
\(
k_2(\gamma_2-1)<1.
\)
Then the Grossmann--Horner tower satisfies property \((P6)\). More precisely,
there exists \(M<\infty\) such that
\[
\bar m_\omega(\Delta_\omega)\le M
\]
for every \(\omega\in\Omega\).
\end{lem}

\begin{proof}
Set
\(r_G:=1/k_2(\gamma_2-1).\)
By the assumption \(k_2(\gamma_2-1)<1\), we have \(r_G>1\).

For the uniform finiteness estimate we use the deterministic worst parameters:
the slowest neutral exponent \(\gamma_2\) and the largest singular exponent
\(k_2\). The deterministic endpoint comparison and the singular inverse
estimate give
\[
m(\delta_n^-(\omega))
\lesssim
n^{-r_G-1},
\qquad
m(\delta_{1,n}^-(\omega))
\lesssim
n^{-r_G-1},
\]
uniformly in \(n\) and \(\omega\).

Using the return time structure,
\[
\bar m_\omega(\Delta_\omega)
=
\sum_{n\ge2} n\,m(\delta_n^-(\omega))
+
\sum_{n\ge1} (n+1)\,m(\delta_{1,n}^-(\omega)).
\]
Therefore
\[
\bar m_\omega(\Delta_\omega)
\lesssim
\sum_{n\ge1} n\,n^{-r_G-1}
=
\sum_{n\ge1} n^{-r_G}.
\]
The last series is finite because \(r_G>1\). Thus \((P6)\) holds.
\end{proof}

\subsubsection{Existence of sample measures and decay of correlations}
\label{ss:GH-correlations}

We apply the random tower results of \cite{BahBosRuz19} to the
Grossmann--Horner tower constructed in Section~\ref{sec:GH-tower}.
Let \(0<\rho<1\) be the contraction factor from
Proposition~\ref{distortion'}.
Since properties \((\mathrm P1)\)--\((\mathrm P7)\) are verified,  
Theorem~4.1 of \cite{BahBosRuz19} yields an \(F\)-equivariant family of
absolutely continuous probability measures
\[
\bar\mu_\omega=\bar\rho_\omega\,\bar m_\omega,
\qquad
(F_\omega)_*\bar\mu_\omega=\bar\mu_{\sigma\omega},
\]
whose densities \(\bar\rho_\omega\) lie in the tower Hölder space with
random Hölder constants \(K_\omega\) satisfying
\(
\mathbb P\{K_\omega>n\}\le C_0 e^{-c_0 n^{\vartheta}}
\)
for some \(C_0,c_0>0\) and \(0<\vartheta<1\).

Set
\(
\mu_\omega=(\pi_\omega)_*\bar\mu_\omega,
\qquad
d\boldsymbol\mu(\omega,x)=d\mu_\omega(x)\,d\mathbb P(\omega).
\)
Because \(\pi_{\sigma\omega}\circ F_\omega=h_\omega\circ\pi_\omega\),
we have \((h_\omega)_*\mu_\omega=\mu_{\sigma\omega}\) for
\(\mathbb P\)-a.e.\ \(\omega\); hence \(\boldsymbol\mu\) is \(G\)-invariant.

\paragraph{Lifted observables.}
Let fibre observables
\(
\varphi=(\varphi_\omega)_{\omega\in\Omega},
\psi=(\psi_\omega)_{\omega\in\Omega}
\)
satisfy the hypotheses of Theorem~\ref{thm3} and define
\(
\bar\varphi_\omega=\varphi_\omega\circ\pi_\omega,\qquad
\bar\psi_\omega=\psi_\omega\circ\pi_\omega,\qquad
\Psi_\omega=\bar\psi_\omega\,\bar\rho_\omega .
\)
Uniform Hölder bounds on \(\psi_\omega\) and exponential contraction give
\(
\bar\psi_\omega\in\mathcal F^1_\rho
\)
with norm independent of \(\omega\).
Because \(\bar\rho_\omega\in\mathcal F^1_\rho\) with norm \(K_\omega\),
the product \(\Psi_\omega\) lies in
\(
\mathcal F^1_{\rho}\!
\)
with random norm \(C_\psi K_\omega\).

\paragraph{Application of the tower theorem.}
Proposition~\ref{Tail'} gives the quenched tail exponent
\[
a=\frac1\zeta, 
\qquad
\zeta=k_2(\gamma_1-1).
\]
Applying Theorem~4.2 of \cite{BahBosRuz19} to
\((\bar\varphi_\omega,\Psi_\omega)\) yields, for every \(\delta>0\),
\[
\bigl|\mathrm{Cor}^{(\mathrm F)}_{n,\omega}(\varphi,\psi)\bigr|
\le
C_\omega C_{\varphi,\psi}\,
n^{-\left(\frac1\zeta-1-\delta\right)},
\]
and the same bound for past correlations. The random prefactor
\(C_\omega\) inherits a stretched exponential tail from \(K_\omega\); hence
there are \(C,u'>0\) and \(0<v'<1\) with
\(
\mathbb P\{C_\omega>n\}\le C e^{-u'n^{v'}}.
\)
This completes the proof of Theorem~\ref{thm3}.

\section{Discrete and uniform parameter laws}\label{sec:DDUD}

In this section we verify the lower drift assumptions for the discrete and
uniform parameter laws. The arrays \(A_{n,k}\) and \(B_{n,k}\) are triangular
arrays whose summands are not, in general, independent, due to the endpoint
cocycle. We therefore compare them from below with one coordinate arrays. These
comparison arrays depend only on \(\omega_{n-k}\), and are independent for fixed
\(n\).

Throughout this section, constants are uniform in \(n,k\), and \(\omega\), and
may change from line to line.

\subsection{The Pikovsky endpoint drift}

Write
\(
\alpha_{n-k}(\omega)=\alpha(\sigma^{n-k}\omega).
\)
The endpoint drift array can be bounded from below as follows.

Comparison of arrays for Pikovsky maps.
\begin{lem}
\label{lem:pik-comparison}
There exist \(k_0\in\mathbb N\) and \(0<d_A<D_A<\infty\) such that, for
\(n\ge k\ge k_0\),
\(
A_{n,k}(\omega)\ge \widetilde A_{n,k}(\omega),
\)
where
\[
\widetilde A_{n,k}(\omega)
=
(\alpha_1-1)d_A
\left(
\frac{c_k(\alpha_1)}
{k^{1/(\alpha_1-1)}}
\right)^{\alpha_{n-k}(\omega)-\alpha_1}
-\frac{\alpha_1(\alpha_1-1)}2D_A^2
\left(
\frac{c_k(\alpha_2)}
{k^{1/(\alpha_2-1)}}
\right)^{2\alpha_{n-k}(\omega)-\alpha_1-1}.
\]
Moreover, for fixed \(n\), the variables
\(\{\widetilde A_{n,k}\}_{k=k_0}^n\) are independent and uniformly bounded.
\end{lem}

\begin{proof}
The endpoint estimates imply that the error terms in the local expansion
converge uniformly to zero on the parameter range. Hence, for \(k\ge k_0\),
the coefficient containing the endpoint correction is bounded above and below
by positive constants. The comparison follows by using the lower bound in the
positive term and the upper bound in the negative term. The resulting array
depends only on \(\alpha(\sigma^{n-k}\omega)\), which gives independence for
fixed \(n\).
\end{proof}

We now give expectation estimates.
\begin{lem}
\label{lem:pik-expectation}
For the comparison array above, we have
\[
\frac1n\sum_{k=k_0}^n\mathbb E_\nu\widetilde A_{n,k}
\to c_{A,\mathrm{DD}}>0
\qquad
\frac{\log n}{n}\sum_{k=k_0}^n
\mathbb E_\nu\widetilde A_{n,k}
\to c_{A,\mathrm{UD}}>0
\]
for the discrete law and the uniform law respectively.
\end{lem}

\begin{proof}
For the discrete law, the atom at \(\alpha_1\) gives a positive constant
contribution to the expectation of the comparison array
\(\widetilde A_{n,k}\). The contributions from the remaining atom and the
negative correction term are polynomially smaller. Hence there exists
\(c_A>0\) such that
\[
\frac1n\sum_{k=1}^n
\mathbb E_\nu(\widetilde A_{n,k})
\ge c_A
\]
for all sufficiently large \(n\).

For the uniform law, the relevant term is
\[
\int_{\alpha_1}^{\alpha_2}
\left(
\frac{c_k(\alpha_1)}
{k^{1/(\alpha_1-1)}}
\right)^{t-\alpha_1}\,dt .
\]
Since
\[
-\log
\frac{c_k(\alpha_1)}
{k^{1/(\alpha_1-1)}}
=
\frac1{\alpha_1-1}\log k+O(1),
\]
the integral is comparable to \(1/\log k\). The negative term is of smaller
order. The stated limits are followed.
\end{proof}

By Lemma~\ref{lem:pik-comparison}, the comparison array
\(\widetilde A_{n,k}\) consists of independent uniformly bounded variables for
fixed \(n\). Hence Hoeffding's inequality gives, for every \(t>0\),
\[
\mathbb P\left(
\left|
\sum_{k=k_0}^n
(\widetilde A_{n,k}-\mathbb E_\nu\widetilde A_{n,k})
\right|>nt(\log n)^{-q}
\right)
\le
2\exp\left(-c\frac{n}{(\log n)^{2q}}\right).
\]
By Lemma~\ref{lem:pik-expectation}, there exists \(c_A>0\) such that
\[
\frac{(\log n)^q}{n}
\sum_{k=k_0}^n
\mathbb E_\nu\widetilde A_{n,k}
\ge 2c_A
\]
for all sufficiently large \(n\). Therefore,
\[
\mathbb P\left\{
\frac{(\log n)^q}{n}
\sum_{k=k_0}^n
\widetilde A_{n,k}<c_A
\right\}
\le
2\exp\left(-c\frac{n}{(\log n)^{2q}}\right).
\]
A Borel--Cantelli argument then gives a random cutoff \(N_A(\omega)\) such that
\[
\frac{(\log n)^q}{n}
\sum_{k=k_0}^n
\widetilde A_{n,k}(\omega)\ge c_A
\]
for all \(n\ge N_A(\omega)\), and
\(
\mathbb P\{N_A>n\}\le Ce^{-un^v}.
\)
Since \(A_{n,k}\ge \widetilde A_{n,k}\), Assumption~\ref{itm:A1} follows. The finitely many terms with \(k<k_0\) are uniformly bounded and can be
absorbed by enlarging the random cutoff by a deterministic amount. Hence
Assumption~\ref{itm:A1} follows.

\subsection{The Grossmann--Horner endpoint drift}

Write
\(
\gamma_{n-k}(\omega)=\gamma(\sigma^{n-k}\omega).
\)
The same argument applies to the Grossmann--Horner array.

We give a comparison of arrays for the Grossmann-Horner maps.
\begin{lem}
\label{lem:GH-comparison}
There exist \(k_0\in\mathbb N\) and \(0<d_B<D_B<\infty\) such that
\(
B_{n,k}(\omega)\ge \widetilde B_{n,k}(\omega),
\)
where
\[
\widetilde B_{n,k}(\omega)
=
d_B
\left(
\frac{c_k(\gamma_1)}
{k^{1/(\gamma_1-1)}}
\right)^{\gamma_{n-k}(\omega)-\gamma_1}
-\frac{\gamma_1-1}{2}D_B^2
\left(
\frac{c_k(\gamma_2)}
{k^{1/(\gamma_2-1)}}
\right)^{2\gamma_{n-k}(\omega)-\gamma_1-1}.
\]
For fixed \(n\), the variables
\(\{\widetilde B_{n,k}\}_{k=k_0}^n\) are independent and uniformly bounded.
\end{lem}

\begin{proof}
The endpoint estimates imply that the correction terms in the local expansion
converge uniformly to zero. Since the parameter range is compact and
\(a\ge a_1>0\), the coefficient in the drift expression is bounded above and
below by positive constants. The comparison follows as in the Pikovsky case.
The resulting array depends only on the coordinate
\(\gamma(\sigma^{n-k}\omega)\), giving independence for fixed \(n\).
\end{proof}

We now give the expectation estimates.
\begin{lem}
\label{lem:GH-expectation}
For the comparison array,
\[
\frac1n\sum_{k=k_0}^n
\mathbb E_\nu\widetilde B_{n,k}
\to c_{B,\mathrm{DD}}>0
\qquad
\frac{\log n}{n}\sum_{k=k_0}^n
\mathbb E_\nu\widetilde B_{n,k}
\to c_{B,\mathrm{UD}}>0
\]
under \ref{itm:GHDD} and under \ref{itm:GHUD} respectively.
\end{lem}

\begin{proof}
Only the \(\gamma\)-coordinate affects the order of the terms. Under \ref{itm:GHDD},
the atom at \(\gamma_1\) gives a positive contribution, while the remaining
terms decay polynomially. Under\ref{itm:GHUD}, the same logarithmic averaging argument
as above gives
\(
\mathbb E_\nu\widetilde B_{n,k}\asymp \frac1{\log k}.
\)
The negative term is of smaller order in both cases.
\end{proof}

The same concentration argument applied to the comparison array
\(\widetilde B_{n,k}\), together with Lemma~\ref{lem:GH-expectation}, yields a
random cutoff \(N_B(\omega)\) with stretched-exponential tail such that
\[
\frac{(\log n)^q}{n}
\sum_{k=k_0}^n
\widetilde B_{n,k}(\omega)\ge c_B
\]
for all \(n\ge N_B(\omega)\). Since
\(B_{n,k}\ge\widetilde B_{n,k}\), the finitely many terms with \(k<k_0\) are uniformly bounded and can be
absorbed by enlarging \(N_B\) by a deterministic amount. Hence
Assumption~\ref{itm:B1} follows.

\subsection{Consequences}

The estimates above verify the Pikovsky drift assumption for the discrete and
uniform laws on \([\alpha_1,\alpha_2]\). Hence the concrete
\ref{itm:DD}/\ref{itm:UD} Pikovsky results follow from the conditional theorem
under Assumption~\ref{itm:A1}.

Similarly, the Grossmann--Horner comparison and expectation estimates verify
Assumption~\ref{itm:B1} for the laws \ref{itm:GHDD}/\ref{itm:GHUD}. This
assumption controls the neutral endpoint recursion, so only the
\(\gamma\)-coordinate appears in the comparison array. The singularity
parameter \(k\) enters later through
\[
|y_n^\pm(\omega)|
\asymp
(1\mp x_{n-1}^{\mp}(\sigma\omega))^{1/k(\omega)}.
\]
Consequently, the return time exponent and the resulting correlation rate are
controlled by
\(
\zeta=k_2(\gamma_1-1),
\)
which is the exponent obtained from Assumption~\ref{itm:B1} and the singular
inverse estimate.

\subsection{Conclusion}

For the discrete and uniform parameter laws, the raw endpoint arrays are
bounded from below by independent one-coordinate comparison arrays. The
expectation estimates and Hoeffding's inequality yield the required lower drift
bounds with stretched exponential random cutoffs. Therefore,
\[
\ref{itm:DD}/\ref{itm:UD}\Longrightarrow\ref{itm:A1},
\qquad
\ref{itm:GHDD}/\ref{itm:GHUD}\Longrightarrow\ref{itm:B1},
\]
and the concrete quenched mixing results follow from the corresponding
conditional theorems.

\bibliography{MuhRuz}
\end{document}